\documentclass{tac}

\usepackage{tikz-cd}
\usetikzlibrary{graphs,decorations.pathmorphing,decorations.markings}
\usepackage[all,cmtip]{xy}
\usepackage{amsmath}
\usepackage[charter]{mathdesign}

\usepackage[colorlinks=true]{hyperref}
\hypersetup{allcolors=[rgb]{0.1,0.1,0.4}}

\author{Jaehyeok Lee, Jae-Suk Park$^*$}

\thanks{
The authors are grateful to Emily Riehl
and anonymous reviewers for their comments, suggestions and careful reading of our paper.
This paper is supported by Basic Science Research Institute Fund, whose NRF grant number is 2021R1A6A1A10042944.
}

\address{Department of Mathematics, POSTECH\\
77 Cheongam-Ro, Nam-Gu, Pohang, Gyeongbuk, Korea 37673
}

\title{Enriched Morita theory of monoids in a closed symmetric monoidal category}

\copyrightyear{2023}


\eaddress{jhlee1228lee@postech.ac.kr\CR jaesuk@postech.ac.kr}

\DeclareFontFamily{U}{dutchcal}{\skewchar\font=45 }
\DeclareFontShape{U}{dutchcal}{m}{n}{<-> s*[1.0] dutchcal-r}{}
\DeclareFontShape{U}{dutchcal}{b}{n}{<-> s*[1.0] dutchcal-b}{}
\DeclareMathAlphabet{\mathlcal}{U}{dutchcal}{m}{n}
\SetMathAlphabet{\mathlcal}{bold}{U}{dutchcal}{b}{n}

\DeclareSymbolFont{cmcal}{OMS}{cmsy}{m}{n}
\SetSymbolFont{cmcal}{bold}{OMS}{cmsy}{b}{n}
\DeclareSymbolFontAlphabet{\mathcal}{cmcal} 

\DeclareMathAlphabet{\mathpzc}{OT1}{pzc}{m}{it}

\DeclareMathOperator{\I}{\mathbb{I}}
\DeclareMathOperator{\Obj}{Obj}
\DeclareMathOperator{\op}{op}
\DeclareMathOperator{\pr}{\prime}

\newtheorem{theorem}{Theorem}[section]
\newtheorem{lemma}[theorem]{Lemma}
\newtheorem{corollary}[theorem]{Corollary}
\newtheorem{proposition}[theorem]{Proposition}

\newtheorem{definition}[theorem]{Definition}
\newtheorem{example}[theorem]{Example}
\newtheorem{remark}[theorem]{Remark}

\numberwithin{equation}{section}

\begin{document}

\maketitle

\begin{abstract}
We develop Morita theory
of monoids in a closed symmetric monoidal category,
in the context of enriched category theory.
\end{abstract}

\section{Introduction}
Let $R$, $R'$ be rings.
The Eilenberg-Watts theorem \cite{Eilenberg1960}, \cite{Watts1960}
states that every cocontinuous functor
$\mathcal{F}:\mathit{Mod}_R\rightarrow\mathit{Mod}_{R'}$
between the categories of right modules
is naturally isomorphic to the functor
$-\otimes_R\text{ }\!\!_RM_{R'}$
of taking tensor product over $R$
for some $(R,R')$-bimodule $\text{ }\!\!_RM_{R'}$.
We say $R$, $R'$ are Morita equivalent if
we have an equivalence of categories between
$\mathit{Mod}_R$ and $\mathit{Mod}_{R'}$.
The main theorem of Morita theory \cite{Morita1958}
states that the following are equivalent:
\begin{itemize}
  \item
  Rings $R$, $R'$ are Morita equivalent;
  
  \item
  There exists a finitely generated projective generator
  $P_{R'}$ in $\mathit{Mod}_{R'}$
  together with an isomorphism of rings
  $R\cong \mathit{End}_{R'}(P_{R'})$;

  \item
  There exists an $(R,R')$-bimodule
  $\text{ }\!\!_RM_{R'}$
  and an $(R',R)$-bimodule
  $\text{ }\!\!_{R'}N_R$
  together with isomorphisms of bimodules
  $\text{ }\!\!_RM_{R'}\otimes_{R'}\text{ }\!\!_{R'}N_R\cong \text{ }\!\!_RR_R$
  and
  $\text{ }\!\!_{R'}N_R\otimes_R\text{ }\!\!_RM_{R'}
  \cong \text{ }\!\!_{R'}R'\text{ }\!\!\!\!_{R'}$.
\end{itemize}

We generalize these results
in the context of enriched category theory.
We begin by establishing the Eilenberg-Watts theorem in an enriched context.
We follow the approach introduced by Mark Hovey in \cite[\textsection 1-2]{Hovey2009}
using tensorial strengths of enriched functors between tensored enriched categories.
After establishing the Eilenberg-Watts theorem,
we provide a theorem which characterizes when an enriched category is equivalent to
the enriched category of right modules over the given monoid of the base category.
As a corollary, we obtain the main result of Morita in enriched context.

The base category that we consider in this paper is a closed symmetric monoidal category
$\mathlcal{C}=(\mathcal{C},\otimes,c,[-,-])$
whose underlying category $\mathcal{C}$ is finitely complete and finitely cocomplete.
Some examples are the closed symmetric monoidal categories
$\mathlcal{Set}/\mathlcal{fSet}/\mathlcal{sSet}$ of small sets/finite sets/simplicial sets,
$\mathlcal{Cat}$ of small categories,
$\mathlcal{Ab}/\mathlcal{fAb}$ of abelian groups/finitely generated abelian groups,
$\mathlcal{Vec}_K/\mathlcal{fVec}_K$ of vector spaces/finite dimensional vector spaces over a field $K$,
$\mathlcal{Mod}_R$/$\widehat{\mathlcal{Mod}}_R$/$\mathlcal{dgMod}_R$
of modules/$L$-complete modules/dg-modules over a commutative ring $R$,
$\mathlcal{CGT}$/$\mathlcal{CGT}_{\!\!*}$ of unbased/based compactly generated topological spaces, 
$\mathlcal{S\!p}^{\Sigma}_{\mathit{CGT}_{\!*}}$ of topological symmetric spectra,
$\mathlcal{CGWH}$/$\mathlcal{CGWH}_{\!\!*}$ of unbased/based compactly generated weakly Hausdorff spaces,
$\mathlcal{Ban}$ of Banach spaces with linear contractions.
Every elementary topos is also an example.

We explain our main ideas and results.
A monoid in $\mathlcal{C}$ is a triple $\mathfrak{b}=(b,u_b,m_b)$
where $b$ is an object in $\mathcal{C}$ and
$u_b$, $m_b$ are the unit, product morphisms in $\mathcal{C}$.
For each monoid $\mathfrak{b}$ in $\mathlcal{C}$,
we denote $\mathpzc{Mod}\!_{\mathfrak{b}}$
as the $\mathlcal{C}$-enriched category of right $\mathfrak{b}$-modules.
We can see $b$ as a right $\mathfrak{b}$-module
which we denote as $b_{\mathfrak{b}}$.
Let $\mathcal{D}$ be a tensored $\mathlcal{C}$-enriched category
whose underlying category $\mathcal{D}_0$ has coequalizers.
For each $\mathlcal{C}$-enriched functor
$\mathcal{F}:\mathpzc{Mod}\!_{\mathfrak{b}}\rightarrow \mathcal{D}$,
the object $\mathcal{F}(b_{\mathfrak{b}})$ in $\mathcal{D}$
is equipped with a left action of $\mathfrak{b}$,
and we have the $\mathlcal{C}$-enriched left adjoint functor
\begin{equation*}
  -\circledast_{\mathfrak{b}}\! \text{ }\!\!_{\mathfrak{b}}\mathcal{F}(b_{\mathfrak{b}})
  :
  \mathpzc{Mod}\!_{\mathfrak{b}}\rightarrow \mathcal{D}
\end{equation*}
of taking tensor product over $\mathfrak{b}$.
We show that there is a canonical $\mathlcal{C}$-enriched natural transformation
\begin{equation} \label{eq Intro lambdaF}
  \lambda^{\mathcal{F}}:
  \xymatrix@C=20pt{
    -\!\circledast_{\mathfrak{b}}\! \text{ }\!\!_{\mathfrak{b}}\mathcal{F}(b_{\mathfrak{b}})
    \ar@2{->}[r]
    &\mathcal{F}:\mathpzc{Mod}\!_{\mathfrak{b}}\rightarrow \mathcal{D}
  }
\end{equation}
associated to $\mathcal{F}:\mathpzc{Mod}\!_{\mathfrak{b}}\rightarrow\mathcal{D}$
(Lemma~\ref{lem EWthm lambdaF}).
This was defined in \cite[Proposition 1.1]{Hovey2009} as an ordinary natural transformation
when $\mathcal{D}=\mathpzc{Mod}\!_{\mathfrak{b}'}$ for another monoid $\mathfrak{b}'$ in $\mathlcal{C}$.
Moreover, we show that the following are equivalent
(Proposition~\ref{prop EWthm lambdaF}):
  \begin{itemize}
    \item
    $\mathcal{F}:\mathpzc{Mod}\!_{\mathfrak{b}}\rightarrow\mathcal{D}$ is a $\mathlcal{C}$-enriched left adjoint;

    \item
    $\mathcal{F}:\mathpzc{Mod}\!_{\mathfrak{b}}\rightarrow\mathcal{D}$
    is $\mathlcal{C}$-enriched cocontinuous;

    \item
    $\mathcal{F}:\mathpzc{Mod}\!_{\mathfrak{b}}\rightarrow\mathcal{D}$ 
    preserves $\mathlcal{C}$-tensors,
    i.e. its tensorial strength is invertible,
    and the underlying functor $\mathcal{F}_0$
    preserves coequalizers;

    \item
    The $\mathlcal{C}$-enriched natural transformation
    $\lambda^{\mathcal{F}}:-\circledast_{\mathfrak{b}}\text{ }\!\!_{\mathfrak{b}}\mathcal{F}(b_{\mathfrak{b}})\Rightarrow \mathcal{F}:\mathpzc{Mod}\!_{\mathfrak{b}}\rightarrow \mathcal{D}$
  in (\ref{eq Intro lambdaF})
    is invertible.
  \end{itemize}
Using this result, we prove the following generalization of the Eilenberg-Watts theorem.
Left $\mathfrak{b}$-module objects in $\mathcal{D}$
are introduced in \textsection\!~\ref{subsec LeftmoduleObj}.

\begin{theorem} \label{thm Intro EWthm}
  Let $\mathfrak{b}$ be a monoid in $\mathlcal{C}$
  and let $\mathcal{D}$ be a tensored $\mathlcal{C}$-enriched category
  whose underlying category $\mathcal{D}_0$ has coequalizers.
  We have a fully faithful left adjoint functor
  \begin{equation}\label{eq Intro EWthm}
    \xymatrix@C=20pt{
      \text{ }\!\!_{\mathfrak{b}}\mathcal{D}
      \ar[r]
      &{\mathlcal{C}\text{-}\mathit{Funct}(\mathpzc{Mod}\!_{\mathfrak{b}},\mathcal{D})}
    }
  \end{equation}
  from the category
  of left $\mathfrak{b}$-modules objects in $\mathcal{D}$
  to the category of $\mathlcal{C}$-enriched functors
  $\mathpzc{Mod}\!_{\mathfrak{b}}\rightarrow \mathcal{D}$.
  The essential image of the left adjoint functor (\ref{eq Intro EWthm}) is the coreflective full subcategory
  $\mathlcal{C}\text{-}\mathit{Funct}_{\mathit{cocon}}(\mathpzc{Mod}\!_{\mathfrak{b}},\mathcal{D})$
  of cocontinuous $\mathlcal{C}$-enriched functors
  $\mathpzc{Mod}\!_{\mathfrak{b}}\rightarrow \mathcal{D}$,
  and we have an adjoint equivalence of categories
  \begin{equation*}
    \xymatrix@C=30pt{
      \text{ }\!\!_{\mathfrak{b}}\mathcal{D}
      \ar@<0.5ex>[r]^-{\simeq}
      &{\mathlcal{C}\text{-}\mathit{Funct}_{\mathit{cocon}}(\mathpzc{Mod}\!_{\mathfrak{b}},\mathcal{D})}
      .
      \ar@<0.5ex>[l]^-{\simeq}
    }      
  \end{equation*}
  The coreflection of a $\mathlcal{C}$-enriched functor
  $\mathcal{F}:\mathpzc{Mod}\!_{\mathfrak{b}}\rightarrow\mathcal{D}$
  is the associated $\mathlcal{C}$-enriched natural transformation
  $\lambda^{\mathcal{F}}$
  in (\ref{eq Intro lambdaF}).
\end{theorem}

Let us explain why Theorem~\ref{thm Intro EWthm}
can be seen as a generalization of the Eilenberg-Watts theorem.
Given another monoid $\mathfrak{b}'$ in $\mathlcal{C}$,
we define a $(\mathfrak{b},\mathfrak{b}')$-bimodule
as a left $\mathfrak{b}$-module object in $\mathpzc{Mod}\!_{\mathfrak{b}'}$
(Definition~\ref{def LeftModuleObj bb'bimodule}).
After substituting $\mathcal{D}=\mathpzc{Mod}\!_{\mathfrak{b}'}$ in Theorem~\ref{thm Intro EWthm},
we obtain the following corollary.

\begin{corollary} \label{cor Intro EWthm}
  Let $\mathfrak{b}$, $\mathfrak{b}'$ be monoids in $\mathlcal{C}$.
  We have an adjoint equivalence of categories
  \begin{equation*}
    \xymatrix@C=30pt{
      \text{ }\!\!_{\mathfrak{b}}\mathpzc{Mod}\!_{\mathfrak{b}'}
      \ar@<0.5ex>[r]^-{\simeq}
      &{\mathlcal{C}\text{-}\mathit{Funct}_{\mathit{cocon}}(\mathpzc{Mod}\!_{\mathfrak{b}},\mathpzc{Mod}\!_{\mathfrak{b}'})}
      \ar@<0.5ex>[l]^-{\simeq}
    }
  \end{equation*}
  from the category
  of $(\mathfrak{b},\mathfrak{b}')$-bimodules
  to the category of cocontinuous $\mathlcal{C}$-enriched functors 
  $\mathpzc{Mod}\!_{\mathfrak{b}}\rightarrow\mathpzc{Mod}\!_{\mathfrak{b}'}$.
\end{corollary}
The original Eilenberg-Watts theorem \cite{Eilenberg1960}, \cite{Watts1960}
states that the functor from left to right in Corollary~\ref{cor Intro EWthm}
is essentially surjective when $\mathlcal{C}=\mathlcal{Ab}$.
This has been generalized to the situation where the target category is a
general tensored $\mathlcal{Ab}$-enriched category by Nyman and Smith \cite{Nyman2008}.
The main result of their article is precisely our Theorem~\ref{thm Intro EWthm}
in the special case $\mathlcal{C}=\mathlcal{Ab}$.
We mention that Corollary~\ref{cor Intro EWthm}
has been discussed online
when $\mathlcal{C}$ is a B\'{e}nabou cosmos.
\footnote{
See
https://mathoverflow.net/questions/159735/in-what-generality-does-eilenberg-watts-hold
and
https://ncatlab.org/nlab/show/Eilenberg-Watts+theorem
for the discussions of Corollary~\ref{cor Intro EWthm}
over a B\'{e}nabou cosmos $\mathlcal{C}$.
}

In the original Eilenberg-Watts theorem, we only assume the cocontinuity of the underlying functor
(i.e., preservation of sums and coequalizers).
In a general $\mathlcal{C}$-enriched setting this is not enough, and we use preservation of $\mathlcal{C}$-tensors which is a more restrictive condition than preservation of sums.
The reason why the weaker assumption is enough in the case of $\mathlcal{C}=\mathlcal{Ab}$
is the following special property of abelian module categories:
any natural transformation between cocontinuous functors out of an abelian module category is invertible
as soon as it is invertible at a projective generator.

Next, we characterize when a $\mathlcal{C}$-enriched category $\mathcal{D}$
is equivalent to $\mathpzc{Mod}\!_{\mathfrak{b}}$.
We say an object $X$ in a $\mathlcal{C}$-enriched category $\mathcal{D}$
is a \emph{$\mathlcal{C}$-enriched compact generator} if 
the $\mathlcal{C}$-enriched Hom functor $\mathcal{D}(X,-):\mathcal{D}\rightarrow\mathcal{C}$
is conservative, preserves $\mathlcal{C}$-tensors
and the underlying functor $\mathcal{D}(X,-)_0$ preserves coequalizers
(Definition~\ref{def Morita Enrichedcompactgen}).

\begin{theorem} \label{thm Intro CharacterizingMb}
  Let $\mathfrak{b}$ be a monoid in $\mathlcal{C}$,
  and let $\mathcal{D}$ be a tensored $\mathlcal{C}$-enriched category
  whose underlying category $\mathcal{D}_0$ has coequalizers.
  Then $\mathcal{D}$ is equivalent to $\mathpzc{Mod}\!_{\mathfrak{b}}$
  as $\mathlcal{C}$-enriched categories if and only if
  there exists a $\mathlcal{C}$-enriched compact generator $X$ in $\mathcal{D}$
  inducing an isomorphism of monoids $f:\mathfrak{b}\cong \mathit{End}_{\mathcal{D}}(X)$ in $\mathlcal{C}$.
\end{theorem}

Using Theorem~\ref{thm Intro EWthm} and Theorem~\ref{thm Intro CharacterizingMb},
we establish the main theorem of Morita theory in enriched context.
We say monoids $\mathfrak{b}$ and $\mathfrak{b}'$ in $\mathlcal{C}$ are \emph{Morita equivalent}
if $\mathpzc{Mod}\!_{\mathfrak{b}}$ and $\mathpzc{Mod}\!_{\mathfrak{b}'}$
are equivalent as $\mathlcal{C}$-enriched categories.

\begin{corollary} \label{cor Intro CosMorita}
  Let $\mathfrak{b}$, $\mathfrak{b}'$ be monoids in $\mathlcal{C}$.
  The following are equivalent:
  \begin{itemize}
    \item[(i)]
    Monoids $\mathfrak{b}$, $\mathfrak{b}'$ in $\mathlcal{C}$ are Morita equivalent;

    \item[(ii)]
    There exists a $\mathlcal{C}$-enriched compact generator $x_{\mathfrak{b}'}$ in $\mathpzc{Mod}\!_{\mathfrak{b}'}$
    together with an isomorphism of monoids $\text{$\mathfrak{b}$}\cong \mathit{End}_{\mathpzc{Mod}\!_{\mathfrak{b}'}}(x_{\mathfrak{b}'})$
    in $\mathlcal{C}$;

    \item[(iii)]
    There exists a $(\mathfrak{b},\mathfrak{b}')$-bimodule $\text{ }\!\!_{\mathfrak{b}}x_{\mathfrak{b}'}$
    and a $(\mathfrak{b}',\mathfrak{b})$-bimodule $\text{ }\!\!_{\mathfrak{b}'}y_{\mathfrak{b}}$
    together with isomorphisms of bimodules
    $\text{ }\!\!_{\mathfrak{b}}x_{\mathfrak{b}'}\circledast_{\mathfrak{b}'}\text{ }\!\!_{\mathfrak{b}'}y_{\mathfrak{b}}
    \cong\text{ }\!\!_{\mathfrak{b}}b_{\mathfrak{b}}$
    and
    $\text{ }\!\!_{\mathfrak{b}'}y_{\mathfrak{b}}\circledast_{\mathfrak{b}}\text{ }\!\!_{\mathfrak{b}}x_{\mathfrak{b}'}
    \cong\text{ }\!\!_{\mathfrak{b}'}b'\text{ }\!\!\!\!_{\mathfrak{b}'}$.
  \end{itemize}
\end{corollary}
If we consider $\mathlcal{C}=\mathlcal{Ab}$
in Corollary~\ref{cor Intro CosMorita},
we recover the original result of Morita.

\section{Enriched Categories}\label{sec EnCat}
We fix a closed symmetric monoidal category
$\mathlcal{C}=(\mathcal{C},\otimes,c,[-,-])$
whose underlying category $\mathcal{C}$ is finitely complete and finitely cocomplete.
We denote objects in $\mathcal{C}$ with small letters.
Let $z$, $x$, $y$ be objects in $\mathcal{C}$.
We have the functor $\otimes:\mathcal{C}\times\mathcal{C}\rightarrow \mathcal{C}$
and the unit object $c$ in $\mathcal{C}$,
together with coherence isomorphisms
\begin{equation}\label{eq EnCat asij}
  \begin{aligned}
    a_{z,x,y}
    &:
    \xymatrix@C=18pt{
      z\!\otimes\! (x\!\otimes\! y)
      \ar[r]^-{\cong}
      &(z\!\otimes\! x)\!\otimes\! y      
      ,
    }
    \\
    s_{x,y}
    &:
    \xymatrix@C=18pt{
      x\!\otimes\! y
      \ar[r]^-{\cong}
      &y\!\otimes\! x
      ,      
    }
  \end{aligned}
  \qquad\quad
  \begin{aligned}
    \imath_x
    &:
    \xymatrix@C=18pt{
      c\!\otimes\! x
      \ar[r]^-{\cong}
      &x
      ,      
    }
    \\
    \jmath_x
    &:
    \xymatrix@C=18pt{
      x\!\otimes\! c
      \ar[r]^-{\cong}
      &x
    }
  \end{aligned}
\end{equation}
in $\mathcal{C}$ that are natural in variables $z$, $x$, $y$.
For each object $x$ in $\mathcal{C}$,
the functor $-\otimes x:\mathcal{C}\rightarrow\mathcal{C}$
admits a right adjoint $[x,-]:\mathcal{C}\rightarrow\mathcal{C}$
and we have a functor $[-,-]:\mathcal{C}^{\op}\times\mathcal{C}\rightarrow\mathcal{C}$.

We refer \cite{Borceux1994}, \cite{Kelly2005} for the basics of enriched category theory.
Let $\mathcal{D}$ be a $\mathlcal{C}$-enriched category and let $X$, $Y$, $Z$ be objects in $\mathcal{D}$.
We denote $\mathcal{D}(X,Y)$ as the Hom object and
  $\I_X:c\rightarrow \mathcal{D}(X,X)$,
  $\mu_{X,Y,Z}:\mathcal{D}(Y,Z)\otimes \mathcal{D}(X,Y)\rightarrow \mathcal{D}(X,Z)$
as the identity, composition morphisms in $\mathcal{C}$.
We denote $\mathcal{D}_0$ as the underlying category of $\mathcal{D}$.
A morphism $X\rightarrow Y$ in $\mathcal{D}$ means
a morphism from $X$ to $Y$ in the underlying category $\mathcal{D}_0$ of $\mathcal{D}$.
We denote $\I_X:X\xrightarrow{\cong}X$
as the identity morphism $\I_X:c\rightarrow \mathcal{D}(X,X)$ of $X$ in $\mathcal{D}$.
For each morphism $l:X\rightarrow Y$ in $\mathcal{D}$,
we have morphisms
$l_{\star}:\mathcal{D}(Z,X)\rightarrow\mathcal{D}(Z,Y)$
and
$l^{\star}:\mathcal{D}(Y,Z)\rightarrow\mathcal{D}(X,Z)$
in $\mathcal{C}$.

The category $\mathcal{C}$ has a canonical $\mathlcal{C}$-enriched category structure
whose Hom objects are given by
$\mathcal{C}(x,y)=[x,y]$.
We identify the underlying category of the $\mathlcal{C}$-enriched category $\mathcal{C}$
with the original category $\mathcal{C}$.

Let $\mathcal{D}'$ be another $\mathlcal{C}$-enriched category.
For each $\mathlcal{C}$-enriched functor
$\alpha:\mathcal{D}\rightarrow\mathcal{D}'$,
we have the underlying functor
$\alpha_0:\mathcal{D}_0\to \mathcal{D}'_0$
and we denote 
$\alpha_{X,Y}:\mathcal{D}(X,Y)\rightarrow \mathcal{D}'(\alpha(X),\alpha(Y))$
as the morphism between Hom objects.
We write $I_{\mathcal{D}}:\mathcal{D}\rightarrow\mathcal{D}$ as the identity $\mathlcal{C}$-enriched functor of $\mathcal{D}$.
Let $\beta:\mathcal{D}\rightarrow\mathcal{D}'$
be another $\mathlcal{C}$-enriched functor from $\mathcal{D}$ to $\mathcal{D}'$.
For each $\mathlcal{C}$-enriched natural transformation
$\xi:\alpha\Rightarrow\beta:\mathcal{D}\rightarrow\mathcal{D}'$,
we have the underlying natural transformation
$\xi_0:\alpha_0\Rightarrow \beta_0:\mathcal{D}_0\rightarrow \mathcal{D}_0'$
whose component at each object $X$ in $\mathcal{D}$ is
$(\xi_0)_X=\xi_X:\alpha(X)\to \beta(X)$.
We denote
$\mathlcal{C}\text{-}\mathit{Funct}(\mathcal{D},\mathcal{D}')$
as the category of $\mathlcal{C}$-enriched functors
from $\mathcal{D}$ to $\mathcal{D}'$.

\subsection{Tensored enriched categories and tensorial strengths}
\label{subsec TenEnCat}
We say a $\mathlcal{C}$-enriched category $\mathcal{D}$ is \emph{tensored} if for each
object $X$ in $\mathcal{D}$,
the $\mathlcal{C}$-enriched Hom functor $\mathcal{D}(X,-):\mathcal{D}\rightarrow\mathcal{C}$
admits a left adjoint $\mathlcal{C}$-enriched functor
$-\circledast X:\mathcal{C}\rightarrow\mathcal{D}$.
We denote the components of the unit, counit
of the $\mathlcal{C}$-enriched adjunction
$-\circledast X\dashv \mathcal{D}(X,-)$
at $z\in \Obj(\mathcal{C})$, $Y\in \Obj(\mathcal{D})$ as
\begin{equation*}
  \xymatrix@C=30pt{
    \mathcal{C}
    \ar@<0.5ex>@/^0.8pc/[r]^-{-\text{ }\!\circledast X}
    &\mathcal{D}
    \ar@<0.5ex>@/^0.8pc/[l]^-{\mathcal{D}(X,-)}
  }
  \qquad\quad
  \mathit{Cv}_{z,X}:
  \xymatrix@C=13pt{
    z
    \ar[r]
    &\mathcal{D}(X,z\!\circledast\! X)        
    ,
  }
  \quad
  \mathit{Ev}_{X,Y}:
  \xymatrix@C=13pt{
    \mathcal{D}(X,Y)\!\circledast\! Y
    \ar[r]
    &X    
    .
  }
\end{equation*}
For each morphism $l:z\circledast X\rightarrow Y$ in $\mathcal{D}$,
we denote the corresponding morphism in $\mathcal{C}$ as
$\bar{l}:z\rightarrow \mathcal{D}(X,Y)$
and call it as the \emph{right adjunct} of $l$.
We have a unique isomorphism
$\imath_X:c\circledast X\xrightarrow[]{\cong}X$ in $\mathcal{D}$
whose right adjunct is the morphism
$\bar{\imath}_X=\I_X:c\rightarrow \mathcal{D}(X,X)$ in $\mathcal{C}$.

Let $\mathcal{D}$, $\mathcal{D}'$ be 
tensored $\mathlcal{C}$-enriched categories
and let $z\in\Obj(\mathcal{C})$, $X\in\Obj(\mathcal{D})$.
For each $\mathlcal{C}$-enriched functor
$\beta:\mathcal{D}\rightarrow\mathcal{D}'$,
the \emph{tensorial strength associated to $\beta$ at $z$, $X$}
is a morphism 
$t^{\beta}_{z,X}:
z\circledast \beta(X)
\rightarrow
\beta(z\circledast X)$
in $\mathcal{D}'$
defined as follows:
\begin{equation*}
  \begin{aligned}
    &t^{\beta}_{z,X}
    :=
    \mathit{Ev}_{\beta(X),\beta(z\circledast X)}
    \circ
    (\beta_{X,z\circledast X}\circledast \I_{\beta(X)})
    \circ
    (\mathit{Cv}_{z,X}\circledast\I_{\beta(X)})
    \\
    &:\!\!
    \xymatrix@C=15pt{
      z\!\circledast\! \beta(X) 
      \ar[r]
      &\mathcal{D}(X,z\!\circledast\! X)\!\circledast\! \beta(X)
      \ar[r]
      &\mathcal{D}'\bigl(\beta(X),\beta(z\!\circledast\! X)\bigr)\!\circledast\! \beta(X)
      \ar[r]
      &\beta(z\!\circledast\! X)
      .
    }
  \end{aligned}
\end{equation*}
We say the $\mathlcal{C}$-functor
$\beta:\mathcal{D}\rightarrow\mathcal{D}'$ \emph{preserves $\mathlcal{C}$-tensors}
if the tensorial strength
$t^{\beta}_{z,X}$
associated to $\beta$
is an isomorphism in $\mathcal{D}'$ for every pair
$z$, $X$.

\begin{example}
  The $\mathlcal{C}$-enriched category $\mathcal{C}$ is tensored.
  Let $z$, $x$, $y\in\Obj(\mathcal{C})$.
  The tensored object of $x$, $y$ in $\mathcal{C}$
  is $x\circledast y=x\otimes y$.
  Moreover,
  \begin{itemize}
    \item 
    the coherence isomorphism $\imath_x:c\otimes x\xrightarrow[]{\cong}x$ in (\ref{eq EnCat asij})
    corresponds to the unique isomorphism $\imath_x:c\circledast x\xrightarrow[]{\cong}x$ in $\mathcal{C}$;
      
    \item 
    the coherence isomorphism
    $a_{z,x,y}:z\otimes(x\otimes y)\xrightarrow[]{\cong}(z\otimes x)\otimes y$
    in (\ref{eq EnCat asij})
    corresponds to the tensorial strength
    $t^{-\circledast y}_{z,x}:z\circledast(x\circledast y)\xrightarrow[]{\cong}(z\circledast x)\circledast y$
    associated to the $\mathlcal{C}$-enriched functor $-\circledast y:\mathcal{C}\rightarrow\mathcal{C}$
    at $z$, $x$.
  \end{itemize}  
\end{example}

Let $x$, $y$ be objects in $\mathcal{C}$.
Throughout this paper, we identify the object $x\otimes y$ in $\mathcal{C}$
with the tensored object $x\circledast y$ in $\mathcal{C}$.
For instance, given a monoid $\mathfrak{b}=(b,u_b,m_b)$ in $\mathlcal{C}$,
we denote the product morphism as
$m_b:b\circledast b\rightarrow b$.

Let $\mathcal{D}$ be a tensored $\mathlcal{C}$-enriched category.
For each object $X$ in $\mathcal{D}$,
the $\mathlcal{C}$-enriched functor
$-\circledast X:\mathcal{C}\rightarrow\mathcal{D}$ preserves $\mathlcal{C}$-tensors.
We denote the associated tensorial strength as
\begin{equation*}
  a_{w,z,X}:=t^{-\circledast X}_{w,z}:
  \xymatrix@C=18pt{
    w\!\circledast\! (z\!\circledast\! X)
    \ar[r]^-{\cong}
    &(w\!\circledast\! z)\!\circledast\! X
    ,
  }
  \qquad
  \forall
  w,
  z \in \Obj(\mathcal{C})
  .
\end{equation*}
We often omit this isomorphism
and simply denote $w\circledast z\circledast X\in\Obj(\mathcal{D})$.

\begin{example}
  Let $\mathfrak{b}=(b, u_b, m_b)$ be a monoid in $\mathlcal{C}$.
  We explain the tensored $\mathlcal{C}$-enriched category
  $\mathpzc{Mod}\!_{\mathfrak{b}}$
  of right $\mathfrak{b}$-modules.
  A \emph{right $\mathfrak{b}$-module}
  is a pair
    $z_{\mathfrak{b}}=
    \bigl(
      z,
      \xymatrix@C=15pt{
        z\circledast b
        \ar[r]^-{\gamma_z}
        &z        
      }\!
    \bigr)$
  of an object $z$ in $\mathcal{C}$,
  and a morphism 
  $\gamma_z:z\circledast b\rightarrow z$ in $\mathcal{C}$
  satisfying the right $\mathfrak{b}$-action relations.
  For instance, we have the right $\mathfrak{b}$-module
    $b_{\mathfrak{b}}
    :=
    \bigl(
      b,
      \xymatrix@C=30pt{
        b\circledast b
        \ar[r]^-{\gamma_b=m_b}
        &b
      }\!
    \bigr)$.
  The Hom object between right $\mathfrak{b}$-modules
  $z_{\mathfrak{b}}=(z,\gamma_z)$ and
  $\tilde{z}_{\mathfrak{b}}=(\tilde{z},\gamma_{\tilde{z}})$
  is given by the equalizer
  \begin{equation}\label{eq EnCat Mbdef}
    \xymatrix@C=60pt{
      \mathpzc{Mod}\!_{\mathfrak{b}}(z_{\mathfrak{b}},\tilde{z}_{\mathfrak{b}})
      \ar@{^{(}->}[r]^-{\mathcal{U}_{z_{\mathfrak{b}},\tilde{z}_{\mathfrak{b}}}}
      &\mathcal{C}(z,\tilde{z})
      \ar@<0.7ex>[r]^-{(\gamma_z)^{\star}}
      \ar@<-0.7ex>[r]_-{(\gamma_{\tilde{z}})_{\star}\circ (-\circledast b)_{z,\tilde{z}}}
      &\mathcal{C}(z\!\circledast\! b,\tilde{z})
      .
    }
  \end{equation}
  The tensored object of $w\in\Obj(\mathcal{C})$ and
  $z_{\mathfrak{b}}\in\Obj(\mathpzc{Mod}\!_{\mathfrak{b}})$
  is the right $\mathfrak{b}$-module
  \begin{equation*}
    w\circledast z_{\mathfrak{b}}=(w\circledast z,\text{ }\gamma_{w\circledast z})    
    ,
    \qquad
    \gamma_{w\circledast z}=\I_w\circledast\gamma_z:
    \xymatrix@C=16pt{
      w\circledast z\!\circledast\! b 
      \ar[r]
      &w\circledast z
      .
    }
  \end{equation*}
  For each right $\mathfrak{b}$-module $z_{\mathfrak{b}}=(z,\gamma_z)$,
  the morphism $\gamma_z:z\circledast b\rightarrow z$ in $\mathcal{C}$
  becomes a morphism
  $\gamma_{z_{\mathfrak{b}}}:z\circledast b_{\mathfrak{b}}\rightarrow z_{\mathfrak{b}}$
  in $\mathpzc{Mod}\!_{\mathfrak{b}}$.
  For instance,
  the morphism $\gamma_b=m_b:b\circledast b\rightarrow b$ in $\mathcal{C}$
  becomes a morphism
  $\gamma_{b_{\mathfrak{b}}}:b\circledast b_{\mathfrak{b}}\rightarrow b_{\mathfrak{b}}$
  in $\mathpzc{Mod}\!_{\mathfrak{b}}$.
  The underlying category $(\mathpzc{Mod}\!_{\mathfrak{b}})_0$ of $\mathpzc{Mod}\!_{\mathfrak{b}}$
  has coequalizers.
  For each right $\mathfrak{b}$-module $z_{\mathfrak{b}}=(z,\gamma_z)$,
  we have the following coequalizer diagram in $(\mathpzc{Mod}\!_{\mathfrak{b}})_0$.
  \begin{equation}\label{eq TenEnCat gammazcoequalizer}
    \xymatrix@C=50pt{
      z\!\circledast\! b\!\circledast\! b_{\mathfrak{b}}
      \ar@<0.5ex>[r]^-{\gamma_z\circledast\I_{b_{\mathfrak{b}}}}
      \ar@<-0.5ex>[r]_-{\I_z\circledast \gamma_{b_{\mathfrak{b}}}}
      &z\!\circledast\! b_{\mathfrak{b}}
      \ar[r]^-{\gamma_{z_{\mathfrak{b}}}}
      &z_{\mathfrak{b}}
    }
  \end{equation}    
\end{example}

Let $\mathfrak{b}$ be a monoid in $\mathlcal{C}$.
We have the forgetful $\mathlcal{C}$-enriched functor
  $\mathcal{U}:
    \mathpzc{Mod}\!_{\mathfrak{b}}\rightarrow \mathcal{C}$
whose morphism on Hom objects is given by the equalizer 
$\mathcal{U}_{z_{\mathfrak{b}},\tilde{z}_{\mathfrak{b}}}
:\mathpzc{Mod}\!_{\mathfrak{b}}(z,\tilde{z})\hookrightarrow \mathcal{C}(z,\tilde{z})$
defined in (\ref{eq EnCat Mbdef}).
The forgetful $\mathlcal{C}$-enriched functor
$\mathcal{U}:\mathpzc{Mod}\!_{\mathfrak{b}}\rightarrow\mathcal{C}$
preserves $\mathlcal{C}$-tensors,
as the associated tensorial strength
at $w\in\Obj(\mathcal{C})$, $z_{\mathfrak{b}}=(z,\gamma_z)\in\Obj(\mathpzc{Mod}\!_{\mathfrak{b}})$
is the identity morphism
$w\circledast z=w\circledast z$ in $\mathcal{C}$.

We introduce basic properties of tensorial strengths without proof.
See \cite[\textsection 3]{SegrtRatkovic2013} for detailed explanations.
\begin{enumerate}
  \item
  Let $\mathcal{D}$, $\mathcal{D}'$ be tensored $\mathlcal{C}$-enriched categories
  and let $w$, $z\in\Obj(\mathcal{C})$, $X\in \Obj(\mathcal{D})$.
  For each $\mathlcal{C}$-enriched functor $\beta:\mathcal{D}\rightarrow\mathcal{D}'$,
  the tensorial strength associated to $\beta$ satisfies the following relations.
  \begin{equation}\label{eq TenEnCat strength funct}
    \hspace*{-0.7cm}
    \vcenter{\hbox{
      \xymatrix@R=20pt@C=15pt{
        c\!\circledast\! \beta(X)
        \ar[r]^-{t^{\beta}_{c,X}}
        \ar@/_1pc/[dr]_-{\imath_{\beta(X)}}^-{\cong}
        &\beta(c\!\circledast\! X)
        \ar[d]^-{\beta(\imath_X)}_-{\cong}
        \\
        \text{ }
        &\beta(X)
      }  
    }}
    \quad
    \vcenter{\hbox{
      \xymatrix@R=15pt@C=28pt{
        w\!\circledast\! \bigl(z\!\circledast\! \beta(X)\bigr)
        \ar[d]_-{a_{w,z,\beta(X)}}^-{\cong}
        \ar[r]^-{\I_w\circledast t^{\beta}_{z,X}}
        &w\!\circledast\! \beta(z\!\circledast\! X)
        \ar[r]^-{t^{\beta}_{w,z\circledast X}}
        &\beta\bigl(w\!\circledast\! (z\!\circledast\! X)\bigr)
        \ar[d]^-{\beta(a_{w,z,X})}_-{\cong}
        \\
        (w\!\circledast\! z)\!\circledast\! \beta(X)
        \ar[rr]^-{t^{\beta}_{w\circledast z,X}}
        &\text{ }
        &\beta\bigl((w\!\circledast\! z)\!\circledast\! X\bigr)
      }  
    }}
  \end{equation}
  Conversely, suppose we have a functor
  $\mathcal{F}_0:\mathcal{D}_0\to \mathcal{D}'_0$
  between the underlying categories of $\mathcal{D}$, $\mathcal{D}'$
  together with a collection of morphisms in $\mathcal{D}'$
  \begin{equation*}
    \bigl\{\text{ }
    t_{z,X}:z\!\circledast\! \mathcal{F}_0(X)\rightarrow \mathcal{F}_0(z\!\circledast\! X)
    \text{ }\big|\text{ }
    z\in\Obj(\mathcal{C})
    ,\text{ }
    X\in\Obj(\mathcal{D})
    \text{ }\bigr\}
  \end{equation*}
  which is natural in variables $z$, $X$ and satisfies the relations (\ref{eq TenEnCat strength funct}).
  Then we have a unique $\mathlcal{C}$-enriched functor $\beta:\mathcal{D}\rightarrow\mathcal{D}'$
  whose underlying functor  $\beta_0$
  is equal to $\mathcal{F}_0$
  and $t^{\beta}_{z,X}=t_{z,X}$ holds for every pair $z$, $X$.

  \item 
  Let $\alpha$, $\beta:\mathcal{D}\rightarrow\mathcal{D}'$ be $\mathlcal{C}$-enriched
  functors between tensored $\mathlcal{C}$-enriched categories $\mathcal{D}$, $\mathcal{D}'$
  and let $z\in\Obj(\mathcal{C})$, $X\in\Obj(\mathcal{D})$.
  For each $\mathlcal{C}$-enriched natural transformation $\xi:\alpha\Rightarrow\beta:\mathcal{D}\rightarrow\mathcal{D}'$,
  we have the following relation.
  \begin{equation}\label{eq TenEnCat strength nat}
    \vcenter{\hbox{
      \xymatrix@R=15pt@C=30pt{
        z\!\circledast\! \alpha(X)
        \ar[d]_-{\I_z\circledast \xi_X}
        \ar[r]^-{t^{\alpha}_{z,X}}
        &\alpha\bigl(z\!\circledast\! X\bigr)
        \ar[d]^-{\xi_{z\circledast X}}
        \\
        z\!\circledast\! \beta(X)
        \ar[r]^-{t^{\beta}_{z,X}}
        &\beta\bigl(z\!\circledast\! X\bigr)
      }
    }}
  \end{equation}
  Conversely, suppose we are given a natural transformation
  $\xi_0:\alpha_0\Rightarrow \beta_0:\mathcal{D}_0\rightarrow \mathcal{D}'_0$
  between the underlying functors $\alpha_0$, $\beta_0$.
  Then $\xi_0$ becomes a $\mathlcal{C}$-enriched natural transformation
  $\xi:\alpha\Rightarrow\beta:\mathcal{D}\rightarrow\mathcal{D}'$
  if and only if it satisfies the relation (\ref{eq TenEnCat strength nat})
  for every pair $z$, $X$.
  This is precisely the correspondence between $\mathlcal{C}$-enriched natural transformations
  and strong natural transformations,
  first introduced by Anders Kock in \cite{Kock1970}.
  It is also explained in \cite{Berger2019}.

  \item 
  Let $\mathcal{D}$, $\mathcal{D}'$, $\mathcal{D}''$ be tensored $\mathlcal{C}$-enriched categories
  and let $\mathcal{D}\xrightarrow{\beta}\mathcal{D}'\xrightarrow{\beta'}\mathcal{D}''$
  be $\mathlcal{C}$-enriched functors.
  The tensorial strength of
  the composition $\beta'\beta:\mathcal{D}\rightarrow\mathcal{D}''$
  at $z\in\Obj(\mathcal{C})$, $X\in\Obj(\mathcal{D})$ is given by
  \begin{equation*}
    t^{\beta'\beta}_{z,X}
    =\beta'(t^{\beta}_{z,X})\circ t^{\beta'}_{z,\beta(X)}
    :
    \xymatrix@C=15pt{
      z\!\circledast\! \beta'\!\beta(X)
      \ar[r]
      &\beta'\bigl(z\!\circledast\! \beta(X)\bigr)
      \ar[r]
      &\beta'\!\beta(z\!\circledast\! X)
      .
    }
  \end{equation*}
\end{enumerate}

\subsection{Left module objects}
\label{subsec LeftmoduleObj}
For the rest of this section,
$\mathfrak{b}=(b,u_b,m_b)$ is a monoid in $\mathlcal{C}$.

\begin{definition}
  Let $\mathcal{D}$ be a tensored $\mathlcal{C}$-enriched category.
  A \emph{left $\mathfrak{b}$-module object in $\mathcal{D}$}
  is a pair
    $\text{ }\!\!_{\mathfrak{b}}X=
    \bigl(
      X,
      \xymatrix@C=18pt{
        b\!\circledast\! X
        \ar[r]^-{\rho_X}
        &X
      }
      \!
    \bigr)$
  of an object $X$ in $\mathcal{D}$,
  and a morphism $\rho_X:b\circledast X\rightarrow X$ in $\mathcal{D}$
  satisfying the left $\mathfrak{b}$-action relations.
  A morphism $\text{ }\!\!_{\mathfrak{b}}X\rightarrow\text{ }\!\!_{\mathfrak{b}}\tilde{X}$
  of left $\mathfrak{b}$-module objects in $\mathcal{D}$
  is a morphism $X\rightarrow \tilde{X}$ in $\mathcal{D}$
  which is compatible with the left $\mathfrak{b}$-action morphisms $\rho_X$, $\rho_{\tilde{X}}$.
  We denote 
  \begin{equation*}
    \text{ }\!\!_{\mathfrak{b}}\mathcal{D}
  \end{equation*}
  as the category of left $\mathfrak{b}$-module objects in $\mathcal{D}$.
  We do not treat $\text{ }\!\!_{\mathfrak{b}}\mathcal{D}$
  as a $\mathlcal{C}$-enriched category.
\end{definition}

Let $X$ be an object in a  tensored $\mathlcal{C}$-enriched category $\mathcal{D}$.
Then the triple
  $\mathit{End}_{\mathcal{D}}(X)
  :=(
    \mathcal{D}(X,X)
    ,\text{ }
    \I_X
    ,\text{ }
    \mu_{X,X,X}
  )$
is a monoid in $\mathlcal{C}$.
For each morphism $\rho_X:b\circledast X\rightarrow X$ in $\mathcal{D}$,
the pair $(X,\rho_X)$ is a left $\mathfrak{b}$-module object in $\mathcal{D}$
if and only if the right adjunct
$\bar{\rho}_X:b\rightarrow \mathcal{D}(X,X)$
of $\rho_X$ becomes a morphism
$\bar{\rho}_X:\mathfrak{b}\rightarrow\mathit{End}_{\mathcal{D}}(X)$
of monoids in $\mathlcal{C}$.

Let $\mathcal{D}$ be a tensored $\mathlcal{C}$-enriched category
and let 
$\text{ }\!\!_{\mathfrak{b}}X=(X,\rho_X)$
be a left $\mathfrak{b}$-module object in $\mathcal{D}$.
Then the $\mathlcal{C}$-enriched Hom functor 
$\mathcal{D}(X,-):\mathcal{D}\rightarrow\mathcal{C}$
factors through the forgetful $\mathlcal{C}$-enriched functor 
$\mathcal{U}:\mathpzc{Mod}\!_{\mathfrak{b}}\rightarrow\mathcal{C}$.
We have a $\mathlcal{C}$-enriched functor
\begin{equation} \label{eq LeftModuleObj T(bX,-)}
  \vcenter{\hbox{
    \xymatrix@R=18pt@C=40pt{
      \mathcal{D}
      \ar[r]^-{\mathcal{D}(\text{ }\!\!_{\mathfrak{b}}X,-)}
      \ar@/_1pc/[dr]_(0.4){\mathcal{D}(X,-)}
      &\mathpzc{Mod}\!_{\mathfrak{b}}
      \ar[d]^-{\mathcal{U}}
      \\
      \text{ }
      &\mathcal{C}
    }  
  }}
  \qquad\qquad
  \mathcal{D}(\text{ }\!\!_{\mathfrak{b}}X,-):
    \mathcal{D}
    \rightarrow
    \mathpzc{Mod}\!_{\mathfrak{b}}   
\end{equation}
which sends each object $Y$ in $\mathcal{D}$
to the right $\mathfrak{b}$-module
$\mathcal{D}(\text{ }\!\!_{\mathfrak{b}}X,Y)=(\mathcal{D}(X,Y),\gamma_{\mathcal{D}(X,Y)})$
whose right $\mathfrak{b}$-action
is given by
  $\gamma_{\mathcal{D}(X,Y)}
  :\!\!
  \xymatrix@C=20pt{
    \mathcal{D}(X,Y)\!\circledast\! b
    \ar[rr]^-{\I_{\mathcal{D}(X,Y)}\circledast\bar{\rho}_X}
    &\text{ }
    &\mathcal{D}(X,Y)\!\circledast\! \mathcal{D}(X,X)
    \ar[r]^-{\mu_{X,X,Y}}
    &\mathcal{D}(X,Y)
    .
  }$

\begin{definition}
  \label{def LeftModuleObj bb'bimodule}
  Let $\mathfrak{b}'=(b',u_{b'},m_{b'})$ be another monoid in $\mathlcal{C}$.
  We define a \emph{$(\mathfrak{b},\mathfrak{b}')$-bimodule}
  $\text{ }\!\!_{\mathfrak{b}}x_{\mathfrak{b}'}$
  as a left $\mathfrak{b}$-module object in 
  the tensored $\mathlcal{C}$-enriched category
  $\mathpzc{Mod}\!_{\mathfrak{b}'}$
  of right $\mathfrak{b}'$-modules.
  Equivalently, it is a pair
    $\text{ }\!\!_{\mathfrak{b}}x_{\mathfrak{b}'}
    =
    \bigl(
      x_{\mathfrak{b}'},
      \xymatrix@C=15pt{
        b\!\circledast\! x_{\mathfrak{b}'}
        \ar[r]^-{\rho_{x_{\mathfrak{b}'}}}
        &x_{\mathfrak{b}'}
      }\!
    \bigr)$
  of a right $\mathfrak{b}'$-module
    $x_{\mathfrak{b}'}=
    \bigl(
      x,
      \xymatrix@C=12pt{
        x\!\circledast\! b'
        \ar[r]^-{\gamma'_x}
        &x
      }
    \bigr)$
  and a morphism
  $\rho_{x_{\mathfrak{b}'}}:b\circledast x_{\mathfrak{b}'}\rightarrow x_{\mathfrak{b}'}$ in $\mathpzc{Mod}\!_{\mathfrak{b}'}$
  satisfying the left $\mathfrak{b}$-action relations.
  We denote
  \begin{equation*}
    \text{ }\!\!_{\mathfrak{b}}\mathpzc{Mod}\!_{\mathfrak{b}'}
  \end{equation*}
  as
  the category of $(\mathfrak{b},\mathfrak{b}')$-bimodules.
  We do not treat $\text{ }\!\!_{\mathfrak{b}}\mathpzc{Mod}\!_{\mathfrak{b}'}$
  as a $\mathlcal{C}$-enriched category.
  Note that we have the $(\mathfrak{b},\mathfrak{b})$-bimodule
    $\text{ }\!\!_{\mathfrak{b}}b_{\mathfrak{b}}
    :=
    (
      b_{\mathfrak{b}}
      ,\text{ }
      \gamma_{b_{\mathfrak{b}}}:
      \!\!
      \xymatrix@C=15pt{
        b\!\circledast\! b_{\mathfrak{b}}\rightarrow b_{\mathfrak{b}}
      }
      \!\!
    )$.
\end{definition}

\begin{example}
  We explain what
  $\mathpzc{Mod}\!_{\mathfrak{b}}$ and
  $\text{ }\!\!_{\mathfrak{b}}\mathpzc{Mod}\!_{\mathfrak{b}'}$
  are in each example of the base category $\mathlcal{C}$.
  \begin{enumerate}
  \item
  Let $\mathlcal{C}=\mathlcal{Ab}$ be the closed symmetric monoidal category of abelian groups.
  \begin{itemize}
    \item 
     Monoids $\mathfrak{b}$, $\mathfrak{b}'$ in $\mathlcal{C}$ are rings;
    
    \item
     $\mathpzc{Mod}\!_{\mathfrak{b}}$ is the preadditive category of right modules over the ring $\mathfrak{b}$;
    
    \item
    $\text{ }\!\!_{\mathfrak{b}}\mathpzc{Mod}\!_{\mathfrak{b}'}$
    is the category of $(\mathfrak{b},\mathfrak{b}')$-bimodules.    
  \end{itemize}

  \item
  Let $\mathlcal{C}=\mathlcal{fAb}$ be the closed symmetric monoidal category of finitely generated abelian groups.
  \begin{itemize}
    \item 
    Monoids $\mathfrak{b}$, $\mathfrak{b}'$ in $\mathlcal{C}$ are rings finitely generated as abelian groups;
    
    \item 
    $\mathpzc{Mod}\!_{\mathfrak{b}}$ is the preadditive category of
    right modules over the ring $\mathfrak{b}$
    which are finitely generated as abelian groups;

    \item 
    $\text{ }\!\!_{\mathfrak{b}}\mathpzc{Mod}\!_{\mathfrak{b}'}$
    is the category of $(\mathfrak{b},\mathfrak{b}')$-bimodules
    which are finitely generated as abelian groups.
  \end{itemize}

  \item 
  Let $\mathlcal{C}=\mathlcal{sSet}$ be the closed symmetric monoidal category of simplicial sets.
  \begin{itemize}
    \item 
    Monoids $\mathfrak{b}$, $\mathfrak{b}'$ in $\mathlcal{C}$ are simplicial monoids;

    \item 
    $\mathpzc{Mod}\!_{\mathfrak{b}}$ is the simplicially enriched category of
    simplicial sets equipped with a right action of the simplicial monoid $\mathfrak{b}$;

    \item
    $\text{ }\!\!_{\mathfrak{b}}\mathpzc{Mod}\!_{\mathfrak{b}'}$
    is the category of simplicial sets equipped with a bi-action of the simplicial monoids $\mathfrak{b}$, $\mathfrak{b}'$.
  
  \end{itemize}

  \item
  Let $\mathlcal{C}=\mathlcal{Ban}$ be the closed symmetric monoidal category of Banach spaces
  and linear contractions between them, equipped with the projective tensor product.
  \begin{itemize}
    \item 
    Monoids $\mathfrak{b}$, $\mathfrak{b}'$ in $\mathlcal{C}$ are associative unital Banach algebras;

    \item
    $\mathpzc{Mod}\!_{\mathfrak{b}}$ is the $\mathlcal{Ban}$-enriched category
    of Banach spaces equipped with a right action of the Banach algebra $\mathfrak{b}$;

    \item
    $\text{ }\!\!_{\mathfrak{b}}\mathpzc{Mod}\!_{\mathfrak{b}'}$
    is the category of Banach spaces equipped with a bi-action of the Banach algebras $\mathfrak{b}$, $\mathfrak{b}'$.
  \end{itemize}

  \item
  Let $\mathlcal{C}=\mathlcal{S\!p}^{\Sigma}_{\mathit{CGT}_{\!*}}$
  be the closed symmetric monoidal category of topological symmetric spectra.
  \begin{itemize}
    \item 
    Monoids $\mathfrak{b}$, $\mathfrak{b}'$ in $\mathlcal{C}$ are symmetric ring spectra;

    \item
    $\mathpzc{Mod}\!_{\mathfrak{b}}$ is the $\mathlcal{S\!p}^{\Sigma}_{\mathit{CGT}_{\!*}}$-enriched category
    of symmetric spectra equipped with a right action of the symmetric ring spectrum $\mathfrak{b}$;

    \item
    $\text{ }\!\!_{\mathfrak{b}}\mathpzc{Mod}\!_{\mathfrak{b}'}$
    is the category of symmetric spectra
    equipped with a bi-action of the symmetric ring spectra $\mathfrak{b}$, $\mathfrak{b}'$.
  \end{itemize}
  \end{enumerate}
\end{example}

\begin{definition}\label{def LeftModuleObj MbbTtoT}
  Let $\mathcal{D}$ be a tensored $\mathlcal{C}$-enriched category
  whose underlying category $\mathcal{D}_0$ has coequalizers.
  We define the functor
  \begin{equation*}
  -\circledast_{\mathfrak{b}}-:
    (\mathpzc{Mod}\!_{\mathfrak{b}})_0\times
    \text{ }\!\!_{\mathfrak{b}}\mathcal{D}
    \rightarrow
   \mathcal{D}_0
  \end{equation*}
  as follows.
  The functor
  sends each pair of a right $\mathfrak{b}$-module $z_{\mathfrak{b}}=(z,\gamma_z)$
  and an object $\text{ }\!\!_{\mathfrak{b}}X=(X,\rho_X)$ in $\text{ }\!\!_{\mathfrak{b}}\mathcal{D}$
  to the following coequalizer in $\mathcal{D}_0$.
  \begin{equation*}
    \xymatrix@C=40pt{
      z\!\circledast\! b\!\circledast\! X
      \ar@<0.5ex>[r]^-{\gamma_z\circledast\I_X}
      \ar@<-0.5ex>[r]_-{\I_z\circledast \rho_X}
      &z\!\circledast\! X
      \ar@{->>}[r]^-{\mathit{cq}_{z_{\mathfrak{b}},\text{ }\!\!_{\mathfrak{b}}\!X}}
      &z_{\mathfrak{b}}\!\circledast_{\mathfrak{b}}\! \text{ }\!\!_{\mathfrak{b}}X
    }
  \end{equation*}
  The functor
  sends each pair of
  a morphism $l:z_{\mathfrak{b}}\rightarrow\tilde{z}_{\mathfrak{b}}$ in $\mathpzc{Mod}\!_{\mathfrak{b}}$
  and a morphism $\tilde{l}:\text{ }\!\!_{\mathfrak{b}}X\rightarrow \text{ }\!\!_{\mathfrak{b}}\tilde{X}$
  in $\text{ }\!\!_{\mathfrak{b}}\mathcal{D}$ to the unique morphism
  $l\circledast_{\mathfrak{b}}\tilde{l}
  :z_{\mathfrak{b}}\circledast_{\mathfrak{b}}\text{ }\!\!_{\mathfrak{b}}X
  \rightarrow \tilde{z}_{\mathfrak{b}}\circledast_{\mathfrak{b}}\text{ }\!\!_{\mathfrak{b}}\tilde{X}$
  in $\mathcal{D}$
  satisfying the relation 
  \begin{equation*}
    \vcenter{\hbox{
      \xymatrix@R=15pt@C=45pt{
        z\!\circledast\! X
        \ar@{->>}[d]_-{\mathit{cq}_{z_{\mathfrak{b}},\text{ }\!\!_{\mathfrak{b}}X}}
        \ar[r]^-{l\circledast \tilde{l}}
        &\tilde{z}\!\circledast\! \tilde{X}
        \ar@{->>}[d]^(0.51){\mathit{cq}_{\tilde{z}_{\mathfrak{b}},\text{ }\!\!_{\mathfrak{b}}\tilde{X}}}
        \\
        z_{\mathfrak{b}}\!\circledast_{\mathfrak{b}}\! \text{ }\!\!_{\mathfrak{b}}X
        \ar@{.>}[r]^-{\exists!\text{ }l\circledast_{\mathfrak{b}}\tilde{l}}
        &\tilde{z}_{\mathfrak{b}}\!\circledast_{\mathfrak{b}}\! \text{ }\!\!_{\mathfrak{b}}\tilde{X}
      }  
    }}
    \qquad
    \mathit{cq}_{\tilde{z}_{\mathfrak{b}},\text{ }\!\!_{\mathfrak{b}}\tilde{X}}
    \circ
    (l\circledast \tilde{l})
    =
    (l\circledast_{\mathfrak{b}} \tilde{l})
    \circ
    \mathit{cq}_{z_{\mathfrak{b}},\text{ }\!\!_{\mathfrak{b}}X}
    .
  \end{equation*}
\end{definition}

Let $\mathcal{D}$ be a tensored $\mathlcal{C}$-enriched category
whose underlying category $\mathcal{D}_0$ has coequalizers.
For each object 
$\text{ }\!\!_{\mathfrak{b}}X=(X,\rho_X)$ in $\text{ }\!\!_{\mathfrak{b}}\mathcal{D}$,
we have a unique isomorphism
$\imath^{\mathfrak{b}}_{\text{ }\!\!_{\mathfrak{b}}X}
:b_{\mathfrak{b}}\circledast_{\mathfrak{b}}\text{ }\!\!_{\mathfrak{b}} X\xrightarrow[]{\cong} X$
in $\mathcal{D}$ which satisfies the relation
$\rho_X=
\imath^{\mathfrak{b}}_{\text{ }\!\!_{\mathfrak{b}}X}
\circ
\mathit{cq}_{b_{\mathfrak{b}},\text{ }\!\!_{\mathfrak{b}}X}$.
The inverse of $\imath^{\mathfrak{b}}_{\text{ }\!\!_{\mathfrak{b}}X}$ is given by
\begin{equation} \label{eq LeftModuleObj imathbbX}
  \vcenter{\hbox{
    \xymatrix@R=15pt@C=50pt{
      b\!\circledast\! X
      \ar@{->>}[d]_-{\mathit{cq}_{b_{\mathfrak{b}},\text{ }\!\!_{\mathfrak{b}}X}}
      \ar@/^1pc/[dr]^-{\rho_X}
      \\
      b_{\mathfrak{b}}\!\circledast_{\mathfrak{b}}\! \text{ }\!\!_{\mathfrak{b}}X
      \ar@{.>}[r]^(0.52){\exists!\text{ }\imath^{\mathfrak{b}}_{\text{ }\!\!_{\mathfrak{b}}X}}_-{\cong}
      &\text{ }X
    }
  }}
  \qquad\quad
  \begin{aligned}
    &
    (\imath^{\mathfrak{b}}_{\text{ }\!\!_{\mathfrak{b}}X})^{-1}
    =
    \mathit{cq}_{b_{\mathfrak{b}},\text{ }\!\!_{\mathfrak{b}}X}
    \circ
    (u_b\!\circledast\! \I_X)
    \circ
    \imath^{-1}_X
    \\
    &\text{ }\text{ }\text{ }:\!\!
    \xymatrix@C=18pt{
      X
      \ar[r]^-{\cong}
      &c\!\circledast\! X
      \ar[r]
      &b\!\circledast\! X
      \ar@{->>}[r]
      &b_{\mathfrak{b}}\!\circledast_{\mathfrak{b}}\! \text{ }\!\!_{\mathfrak{b}}X
      .
    }
  \end{aligned}
\end{equation}

\begin{lemma}
  \label{lem LeftModuleObj Leftadj}
  Let $\mathcal{D}$ be a tensored $\mathlcal{C}$-enriched category
  whose underlying category $\mathcal{D}_0$ has coequalizers.
  We have a well-defined functor
  \begin{equation}\label{eq LeftModuleObj Leftadj}
    \vcenter{\hbox{
      \xymatrix@R=0pt{
        \text{ }\!\!_{\mathfrak{b}}\mathcal{D}
        \ar[r]
        &{\mathlcal{C}\text{-}\mathit{Funct}(\mathpzc{Mod}\!_{\mathfrak{b}},\mathcal{D})}
        \\
        \text{ }\!\!_{\mathfrak{b}}X
        \ar@{|->}[r]
        &-\circledast_{\mathfrak{b}}\! \text{ }\!\!_{\mathfrak{b}}X:\mathpzc{Mod}\!_{\mathfrak{b}}\rightarrow\mathcal{D}
      }  
    }}
  \end{equation}
  from the category of left $\mathfrak{b}$-module objects in $\mathcal{D}$
  to the category of $\mathlcal{C}$-enriched functors
  $\mathpzc{Mod}\!_{\mathfrak{b}}\rightarrow \mathcal{D}$.
  \begin{enumerate}
    \item 
    For each object $\text{ }\!\!_{\mathfrak{b}}X=(X,\rho_X)$ in $\text{ }\!\!_{\mathfrak{b}}\mathcal{D}$,
    we have a $\mathlcal{C}$-enriched functor
    $-\circledast_{\mathfrak{b}}\text{ }\!\!_{\mathfrak{b}}X:\mathpzc{Mod}\!_{\mathfrak{b}}\rightarrow\mathcal{D}$
    which is uniquely determined as follows.
    The underlying functor of
    $-\circledast_{\mathfrak{b}}\text{ }\!\!_{\mathfrak{b}}X$
    is defined in Definition~\ref{def LeftModuleObj MbbTtoT},
    and the associated tensorial strength 
    \begin{equation*}
      \hspace*{-0.3cm}
      a_{w,z_{\mathfrak{b}},\text{ }\!\!_{\mathfrak{b}}X}
      :=t^{-\circledast_{\mathfrak{b}}\text{ }\!\!_{\mathfrak{b}}X}_{w,z_{\mathfrak{b}}}
      :
      \xymatrix@C=15pt{
        w\!\circledast\! (z_{\mathfrak{b}}\!\circledast_{\mathfrak{b}}\! \text{ }\!\!_{\mathfrak{b}}X)
        \ar[r]^-{\cong}
        &(w\!\circledast\! z_{\mathfrak{b}})\!\circledast_{\mathfrak{b}}\! \text{ }\!\!_{\mathfrak{b}}X
        ,
      }
      \quad
      w\in\Obj(\mathcal{C}),
      \text{ }
      z_{\mathfrak{b}}\in\Obj(\mathpzc{Mod}\!_{\mathfrak{b}})
    \end{equation*}
    is the unique isomorphism in $\mathcal{D}$ satisfying the relation
    \begin{equation*}
      \vcenter{\hbox{
        \xymatrix@R=15pt@C=50pt{
          w\!\circledast\! (z\!\circledast\! X)
          \ar@{->>}[d]_-{\I_w\circledast \mathit{cq}_{z_{\mathfrak{b}},\text{ }\!\!_{\mathfrak{b}}X}}
          \ar[r]^-{a_{w,z,X}}_-{\cong}
          &(w\!\circledast\! z)\!\circledast\! X
          \ar@{->>}[d]^-{\mathit{cq}_{w\circledast z_{\mathfrak{b}},\text{ }\!\!_{\mathfrak{b}}X}}
          \\
          w\!\circledast\! (z_{\mathfrak{b}}\!\circledast_{\mathfrak{b}}\! \text{ }\!\!_{\mathfrak{b}}X)
          \ar@{.>}[r]^-{\exists!\text{ }a_{w,z_{\mathfrak{b}},\text{ }\!\!_{\mathfrak{b}}X}}_-{\cong}
          &(w\!\circledast\! z_{\mathfrak{b}})\!\circledast_{\mathfrak{b}}\! \text{ }\!\!_{\mathfrak{b}}X
        }  
      }}
      \qquad
      \begin{aligned}
        &
        \mathit{cq}_{w\circledast z_{\mathfrak{b}},\text{ }\!\!_{\mathfrak{b}}X}
        \circ
        a_{w,z,X}
        \\
        &=
        a_{w,z_{\mathfrak{b}},\text{ }\!\!_{\mathfrak{b}}X}
        \circ
        (\I_w\circledast\mathit{cq}_{z_{\mathfrak{b}},\text{ }\!\!_{\mathfrak{b}}X})
        .
      \end{aligned}
    \end{equation*}

    \item 
    For each morphism $\text{ }\!\!_{\mathfrak{b}}X\rightarrow\text{ }\!\!_{\mathfrak{b}}\tilde{X}$
    in $\text{ }\!\!_{\mathfrak{b}}\mathcal{D}$,
    the following collection of morphisms in $\mathcal{D}$
    \begin{equation} \label{eq2 LeftModuleObj awzbbX}
      \bigl\{\text{ }
      z_{\mathfrak{b}}\!\circledast_{\mathfrak{b}}\!\text{ }\!\!_{\mathfrak{b}}X
      \rightarrow z_{\mathfrak{b}}\!\circledast_{\mathfrak{b}}\!\text{ }\!\!_{\mathfrak{b}}\tilde{X}        
      \text{ }\big|\text{ }
      z_{\mathfrak{b}}\in\Obj(\mathpzc{Mod}\!_{\mathfrak{b}})
      \text{ }\bigr\}
    \end{equation}
    defines a $\mathlcal{C}$-enriched natural transformation
    $-\circledast_{\mathfrak{b}}\text{ }\!\!_{\mathfrak{b}}X
    \Rightarrow-\circledast_{\mathfrak{b}}\text{ }\!\!_{\mathfrak{b}}\tilde{X}
    :\mathpzc{Mod}\!_{\mathfrak{b}}\rightarrow\mathcal{D}$.
  \end{enumerate}
\end{lemma}
\begin{proof}
  We leave for the readers to check that
  such isomorphisms $a_{w,z_{\mathfrak{b}},\text{ }\!\!_{\mathfrak{b}}X}$
  in $\mathcal{D}$ uniquely exist,
  and satisfy the relations (\ref{eq TenEnCat strength funct}).
  Thus we have a unique $\mathlcal{C}$-enriched functor
  $-\circledast_{\mathfrak{b}}\text{ }\!\!_{\mathfrak{b}}X:\mathpzc{Mod}\!_{\mathfrak{b}}\rightarrow\mathcal{D}$
  as described in statement 1.
  Statement 2 is also true,
  as we can check that the collection (\ref{eq2 LeftModuleObj awzbbX}) of morphisms in $\mathcal{D}$
  satisfies the relation (\ref{eq TenEnCat strength nat}).
  We conclude that the functor (\ref{eq LeftModuleObj Leftadj}) is well-defined.
\end{proof}

We will show in \textsection~\!\ref{sec EWthm}
that the functor (\ref{eq LeftModuleObj Leftadj}) in Lemma~\ref{lem LeftModuleObj Leftadj}
is the fully faithful left adjoint functor (\ref{eq Intro EWthm})
described in Theorem~\ref{thm Intro EWthm}.

\begin{proposition} \label{prop LeftModuleObj bX Cadj}
  Let $\mathcal{D}$ be a tensored $\mathlcal{C}$-enriched category
  whose underlying category $\mathcal{D}_0$ has coequalizers.
  For each left $\mathfrak{b}$-module object $\text{ }\!\!_{\mathfrak{b}}X$ in $\mathcal{D}$,
  we have a $\mathlcal{C}$-enriched adjunction
  \begin{equation*}
    \xymatrix@C=45pt{
      \mathpzc{Mod}\!_{\mathfrak{b}}
      \ar@<0.5ex>@/^0.8pc/[r]^-{-\text{ }\!\circledast_{\mathfrak{b}}\text{ }\!\!_{\mathfrak{b}}X}
      &\mathcal{D}
      \ar@<0.5ex>@/^0.8pc/[l]^-{\mathcal{D}(\text{ }\!\!_{\mathfrak{b}}X,-)}
    }
    \qquad\quad
    -\circledast_{\mathfrak{b}}\text{ }\!\!_{\mathfrak{b}}X
    \dashv
    \mathcal{D}(\text{ }\!\!_{\mathfrak{b}}X,-)
    :\mathpzc{Mod}\!_{\mathfrak{b}}\rightarrow\mathcal{D}
  \end{equation*}
  whose unit, counit $\mathlcal{C}$-enriched natural transformations
  $\eta$, $\varepsilon$
  are described below.
  \begin{itemize}
    \item
    The component of the unit $\eta$
    at each $z_{\mathfrak{b}}=(z,\gamma_z)\in\Obj(\mathpzc{Mod}\!_{\mathfrak{b}})$
    is the unique morphism
    $\eta_{z_{\mathfrak{b}}}: z_{\mathfrak{b}} \rightarrow
    \mathcal{D}(\text{ }\!\!_{\mathfrak{b}}X,z_{\mathfrak{b}}\!\circledast_{\mathfrak{b}}\! \text{ }\!\!_{\mathfrak{b}}X)$
    in $\mathpzc{Mod}\!_{\mathfrak{b}}$,
    whose corresponding morphism in $\mathcal{C}$ is
    \begin{equation*}
      \eta_{z_{\mathfrak{b}}}:
      \xymatrix@C=40pt{
        z
        \ar[r]^-{\mathit{Cv}_{z,X}}
        &\mathcal{D}(X,z\!\circledast\! X)
        \ar[r]^-{(\mathit{cq}_{z_{\mathfrak{b}},\text{ }\!\!_{\mathfrak{b}}X})_{\star}}
        &\mathcal{D}(X,z_{\mathfrak{b}}\!\circledast_{\mathfrak{b}}\! \text{ }\!\!_{\mathfrak{b}}X)
        .
      }
    \end{equation*}

    \item
    The component of the counit $\varepsilon$
    at each $Y\in\Obj(\mathcal{D})$
    is the unique morphism
    $\varepsilon_Y:\mathcal{D}(\text{ }\!\!_{\mathfrak{b}}X,Y)\circledast_{\mathfrak{b}}\text{ }\!\!_{\mathfrak{b}}X\rightarrow Y$ in $\mathcal{D}$
    which satisfies the relation
    \begin{equation*}
      \vcenter{\hbox{
        \xymatrix@R=15pt@C=50pt{
          \mathcal{D}(X,Y)\!\circledast\! X
          \ar@{->>}[d]_-{\mathit{cq}_{\mathcal{D}(\text{ }\!\!_{\mathfrak{b}}X,Y),\text{ }\!\!_{\mathfrak{b}}X}}
          \ar@/^1.2pc/[dr]^(0.7){\mathit{Ev}_{X,Y}}
          \\
          \mathcal{D}(\text{ }\!\!_{\mathfrak{b}}X,Y)\!\circledast_{\mathfrak{b}}\! \text{ }\!\!_{\mathfrak{b}}X
          \ar@{.>}[r]^(0.6){\exists!\text{ }\varepsilon_{Y}}
          &Y
        }  
      }}
      \qquad\quad
      \mathit{Ev}_{X,Y}
      =
      \varepsilon_{Y}
      \circ
      \mathit{cq}_{\mathcal{D}(\text{ }\!\!_{\mathfrak{b}}X,Y),\text{ }\!\!_{\mathfrak{b}}X}
      .
    \end{equation*}
  \end{itemize}
\end{proposition}
\begin{proof}
  The components $\eta_{z_{\mathfrak{b}}}$, $\varepsilon_Y$
  are well-defined and are natural in variables $z_{\mathfrak{b}}$, $Y$, respectively.
  As their components $\eta_{z_{\mathfrak{b}}}$, $\varepsilon_Y$
  satisfy the relation (\ref{eq TenEnCat strength nat}),
  we obtain $\mathlcal{C}$-enriched natural transformations
  $\eta$, $\varepsilon$.
  We leave for the readers to check that $\eta$, $\varepsilon$ satisfy the triangular identities.
\end{proof}

\section{The Eilenberg-Watts Theorem} \label{sec EWthm}
In this section, we prove 
Theorem~\ref{thm Intro EWthm}
which generalizes the Eilenberg-Watts theorem in enriched context.
We also give a proof of Corollary~\ref{cor Intro EWthm}.
Throughout this section,
$\mathfrak{b}=(b,u_b,m_b)$ is a monoid in $\mathlcal{C}$
and $\mathcal{D}$ is a tensored $\mathlcal{C}$-enriched category
whose underlying category $\mathcal{D}_0$ has coequalizers.  
We are going to show that the functor
\begin{equation*}
  \vcenter{\hbox{
    \xymatrix@R=0pt@C=30pt{
      \text{ }\!\!_{\mathfrak{b}}\mathcal{D}
      \ar[r]^-{(\ref{eq LeftModuleObj Leftadj})}
      &{\mathlcal{C}\text{-}\mathit{Funct}(\mathpzc{Mod}\!_{\mathfrak{b}},\mathcal{D})}
      \\
      \text{ }\!\!_{\mathfrak{b}}X
      \ar@{|->}[r]
      &-\circledast_{\mathfrak{b}}\! \text{ }\!\!_{\mathfrak{b}}X:\mathpzc{Mod}\!_{\mathfrak{b}}\rightarrow\mathcal{D}
    }  
  }}
\end{equation*}
defined in Lemma~\ref{lem LeftModuleObj Leftadj}
is left adjoint to the functor of evaluating at
$b_{\mathfrak{b}}\in\Obj(\mathpzc{Mod}\!_{\mathfrak{b}})$.
Let us explain the right adjoint functor in detail.
Using the properties (\ref{eq TenEnCat strength funct}), (\ref{eq TenEnCat strength nat})
of tensorial strengths associated to $\mathlcal{C}$-enriched functors
$\mathpzc{Mod}\!_{\mathfrak{b}}\rightarrow\mathcal{D}$,
one can check that the following are true.
\begin{itemize}
  \item 
  For each $\mathlcal{C}$-enriched functor $\mathcal{F}:\mathpzc{Mod}\!_{\mathfrak{b}}\rightarrow\mathcal{D}$,
  the object $\mathcal{F}(b_{\mathfrak{b}})$ in $\mathcal{D}$
  becomes a left $\mathfrak{b}$-module object
  $\text{ }\!\!_{\mathfrak{b}}\mathcal{F}(b_{\mathfrak{b}})
  =\bigl(\mathcal{F}(b_{\mathfrak{b}}),\rho_{\mathcal{F}(b_{\mathfrak{b}})}\bigr)$
  in $\mathcal{D}$ whose left $\mathfrak{b}$-action morphism is
  \begin{equation*}
    \rho_{\mathcal{F}(b_{\mathfrak{b}})}:
    \xymatrix@C=30pt{
      b\!\circledast\! \mathcal{F}(b_{\mathfrak{b}})
      \ar[r]^-{t^{\mathcal{F}}_{b,b_{\mathfrak{b}}}}
      &\mathcal{F}(b\!\circledast\! b_{\mathfrak{b}})
      \ar[r]^-{\mathcal{F}(\gamma_{b_{\mathfrak{b}}})}
      &\mathcal{F}(b_{\mathfrak{b}})
      .
    }
  \end{equation*}

  \item
  For each $\mathlcal{C}$-enriched natural transformation
  $\xi:\mathcal{F}\Rightarrow\widetilde{\mathcal{F}}:\mathpzc{Mod}\!_{\mathfrak{b}}\rightarrow\mathcal{D}$,
  the component
  $\xi_{b_{\mathfrak{b}}}:\mathcal{F}(b_{\mathfrak{b}})\rightarrow \widetilde{\mathcal{F}}(b_{\mathfrak{b}})$
  of $\xi$ at $b_{\mathfrak{b}}$
  becomes a morphism
  $\xi_{b_{\mathfrak{b}}}:
  \text{ }\!\!_{\mathfrak{b}}\mathcal{F}(b_{\mathfrak{b}})
  \rightarrow \text{ }\!\!_{\mathfrak{b}}\widetilde{\mathcal{F}}(b_{\mathfrak{b}})$
  of left $\mathfrak{b}$-module objects in $\mathcal{D}$.
\end{itemize}
Thus we obtain a well-defined functor
\begin{equation}\label{eq EWthm rightCadj}
  \vcenter{\hbox{
    \xymatrix@R=0pt{
      {\mathlcal{C}\text{-}\mathit{Funct}(\mathpzc{Mod}\!_{\mathfrak{b}},\mathcal{D})}
      \ar[r]
      &\text{ }\!\!_{\mathfrak{b}}\mathcal{D}
      \\
      \mathcal{F}:\mathpzc{Mod}\!_{\mathfrak{b}}\rightarrow\mathcal{D}
      \text{ }\ar@{|->}[r]
      &\text{ }\!\!_{\mathfrak{b}}\mathcal{F}(b_{\mathfrak{b}})
    }  
  }}
\end{equation}
of evaluating at $b_{\mathfrak{b}}$.

\begin{lemma} \label{lem EWthm lambdaF}
  For each $\mathlcal{C}$-enriched functor $\mathcal{F}:\mathpzc{Mod}\!_{\mathfrak{b}}\rightarrow\mathcal{D}$,
  we have a $\mathlcal{C}$-enriched natural transformation
  $\lambda^{\mathcal{F}}:
  -\!\circledast_{\mathfrak{b}}\! \text{ }\!\!_{\mathfrak{b}}\mathcal{F}(b_{\mathfrak{b}})
  \Rightarrow \mathcal{F}:\mathpzc{Mod}\!_{\mathfrak{b}}\rightarrow\mathcal{D}$
  whose component $\lambda^{\mathcal{F}}_{z_{\mathfrak{b}}}$
  at $z_{\mathfrak{b}}\in\Obj(\mathpzc{Mod}\!_{\mathfrak{b}})$ is the unique morphism in $\mathcal{D}$
  satisfying the relation
  \begin{equation*}
    \vcenter{\hbox{
      \xymatrix@R=15pt@C=30pt{
        z\!\circledast\! \mathcal{F}(b_{\mathfrak{b}})
        \ar@{->>}[d]_-{\mathit{cq}_{z_{\mathfrak{b}},\text{ }\!\!_{\mathfrak{b}}\mathcal{F}(b_{\mathfrak{b}})}}
        \ar[r]^-{t^{\mathcal{F}}_{z,b_{\mathfrak{b}}}}
        &\mathcal{F}(z\!\circledast\! b_{\mathfrak{b}})
        \ar[d]^-{\mathcal{F}(\gamma_{z_{\mathfrak{b}}})}
        \\
        z_{\mathfrak{b}}\!\circledast_{\mathfrak{b}}\! \text{ }\!\!_{\mathfrak{b}}\mathcal{F}(b_{\mathfrak{b}})
        \ar@{.>}[r]^(0.53){\exists!\text{ }\lambda^{\mathcal{F}}_{z_{\mathfrak{b}}}}
        &\mathcal{F}(z_{\mathfrak{b}})
      }  
    }}
    \qquad\quad
    \mathcal{F}(\gamma_{z_{\mathfrak{b}}})
    \circ t^{\mathcal{F}}_{z,b_{\mathfrak{b}}}
    =
    \lambda^{\mathcal{F}}_{z_{\mathfrak{b}}}
    \circ \mathit{cq}_{z_{\mathfrak{b}},\text{ }\!\!_{\mathfrak{b}}\mathcal{F}(b_{\mathfrak{b}})}
    .
  \end{equation*}
  Moreover, the component of $\lambda^{\mathcal{F}}$ at $b_{\mathfrak{b}}$ is given by
    $\lambda^{\mathcal{F}}_{b_{\mathfrak{b}}}
    =\imath^{\mathfrak{b}}_{\text{ }\!\!_{\mathfrak{b}}\mathcal{F}(b_{\mathfrak{b}})}:
    \xymatrix@C=18pt{
      b_{\mathfrak{b}}\!\circledast_{\mathfrak{b}}\! \text{ }\!\!_{\mathfrak{b}}\mathcal{F}(b_{\mathfrak{b}})
      \ar[r]^-{\cong}
      &\mathcal{F}(b_{\mathfrak{b}})
      .
    }$
\end{lemma}
\begin{proof}
  We leave for the readers to check that 
  such morphism $\lambda^{\mathcal{F}}_{z_{\mathfrak{b}}}$ in $\mathcal{D}$ uniquely exists,
  and that the following diagram of morphisms in $\mathcal{D}$ commutes.
  \begin{equation}\label{eq EWthm lambdaF}
    \vcenter{\hbox{
      \xymatrix@R=20pt@C=70pt{
        z\!\circledast\! b\!\circledast\! \mathcal{F}(b_{\mathfrak{b}})
        \ar@<0.5ex>[d]^-{\I_z\circledast\rho_{\mathcal{F}(b_{\mathfrak{b}})}}
        \ar@<-0.5ex>[d]_-{\gamma_z\circledast\I_{\mathcal{F}(b_{\mathfrak{b}})}}
        \ar[r]^-{t^{\mathcal{F}}_{z\circledast b,b_{\mathfrak{b}}}}
        &\mathcal{F}(z\!\circledast\! b\!\circledast\! b_{\mathfrak{b}})
        \ar@<0.5ex>[d]^-{\mathcal{F}(\I_z\circledast \gamma_{b_{\mathfrak{b}}})}
        \ar@<-0.5ex>[d]_-{\mathcal{F}(\gamma_z\circledast\I_{b_{\mathfrak{b}}})}
        \\
        z\!\circledast\! \mathcal{F}(b_{\mathfrak{b}})
        \ar@{->>}[d]_-{\mathit{cq}_{z_{\mathfrak{b}},\text{ }\!\!_{\mathfrak{b}}\mathcal{F}(b_{\mathfrak{b}})}}
        \ar[r]^-{t^{\mathcal{F}}_{z,b_{\mathfrak{b}}}}
        &\mathcal{F}(z\!\circledast\! b_{\mathfrak{b}})
        \ar[d]^-{\mathcal{F}(\gamma_{z_{\mathfrak{b}}})}
        \\
        z_{\mathfrak{b}}\!\circledast_{\mathfrak{b}}\! \text{ }\!\!_{\mathfrak{b}}\mathcal{F}(b_{\mathfrak{b}})
        \ar@{.>}[r]^-{\exists!\text{ }\lambda^{\mathcal{F}}_{z_{\mathfrak{b}}}}
        &\mathcal{F}(z_{\mathfrak{b}})
      }
    }}
  \end{equation}
  The collection $\{\lambda^{\mathcal{F}}_{z_{\mathfrak{b}}}\}$ of morphisms in $\mathcal{D}$
  is natural in variable $z_{\mathfrak{b}}$.
  To show that the collection $\{\lambda^{\mathcal{F}}_{z_{\mathfrak{b}}}\}$
  is $\mathlcal{C}$-enriched natural in variable $z_{\mathfrak{b}}$,
  we need to verify the following relation for every pair $w\in\Obj(\mathcal{C})$,
  $z_{\mathfrak{b}}\in\Obj(\mathpzc{Mod}\!_{\mathfrak{b}})$.
  \begin{equation}\label{eq2 EWthm lambdaF}
    \vcenter{\hbox{
      \xymatrix@R=15pt@C=50pt{
        w\!\circledast\! \bigl(z_{\mathfrak{b}}\!\circledast_{\mathfrak{b}}\! \text{ }\!\!_{\mathfrak{b}}\mathcal{F}(b_{\mathfrak{b}})\bigr)
        \ar[d]_-{\I_w\circledast \lambda^{\mathcal{F}}_{z_{\mathfrak{b}}}}
        \ar[r]^-{a_{w,z_{\mathfrak{b}},\text{ }\!\!_{\mathfrak{b}}\mathcal{F}(b_{\mathfrak{b}})}}_-{\cong}
        &(w\!\circledast\! z_{\mathfrak{b}})\!\circledast_{\mathfrak{b}}\! \text{ }\!\!_{\mathfrak{b}}\mathcal{F}(b_{\mathfrak{b}})
        \ar[d]^-{\lambda^{\mathcal{F}}_{w\circledast z_{\mathfrak{b}}}}
        \\
        w\!\circledast\! \mathcal{F}(z_{\mathfrak{b}})
        \ar[r]^-{t^{\mathcal{F}}_{w,z_{\mathfrak{b}}}}
        &\mathcal{F}(w\!\circledast\! z_{\mathfrak{b}})
      }
    }}
  \end{equation}
  Consider the following commutative diagram.
  \begin{equation*}
    \hspace*{-1cm}
    \xymatrix@R=12pt@C=25pt{
      &w\!\circledast\! \bigl(z\!\circledast\! \mathcal{F}(b_{\mathfrak{b}})\bigr)
      \ar@{->>}[d]^-{\textcolor{teal}{\I_w\circledast \mathit{cq}_{z_{\mathfrak{b}},\text{ }\!\!_{\mathfrak{b}}\mathcal{F}(b_{\mathfrak{b}})}}}
      \ar@{=}[r]
      &w\!\circledast\! \bigl(z\!\circledast\! \mathcal{F}(b_{\mathfrak{b}})\bigr)
      \ar[dd]^-{a_{w,z,\mathcal{F}(b_{\mathfrak{b}})}}_-{\cong}
      \ar@{=}[r]
      &w\!\circledast\! \bigl(z\!\circledast\! \mathcal{F}(b_{\mathfrak{b}})\bigr)
      \ar[d]^-{\I_w\circledast t^{\mathcal{F}}_{z,b_{\mathfrak{b}}}}
      \ar@{=}[r]
      &w\!\circledast\! \bigl(z\!\circledast\! \mathcal{F}(b_{\mathfrak{b}})\bigr)
      \ar@{->>}[d]^-{\textcolor{teal}{\I_w\circledast \mathit{cq}_{z_{\mathfrak{b}},\text{ }\!\!_{\mathfrak{b}}\mathcal{F}(b_{\mathfrak{b}})}}}
      \\
      &w\!\circledast\! \bigl(z_{\mathfrak{b}}\!\circledast_{\mathfrak{b}}\! \text{ }\!\!_{\mathfrak{b}}\mathcal{F}(b_{\mathfrak{b}})\!\bigr)
      \ar[dd]^-{a_{w,z_{\mathfrak{b}},\text{ }\!\!_{\mathfrak{b}}\mathcal{F}(b_{\mathfrak{b}})}}_-{\cong}
      &\text{ }
      &w\!\circledast\! \mathcal{F}(z\!\circledast\! b_{\mathfrak{b}})
      \ar[d]^-{t^{\mathcal{F}}_{w,z\circledast b_{\mathfrak{b}}}}
      \ar@/^0.5pc/[dr]|-{\I_w\circledast \mathcal{F}(\gamma_{z_{\mathfrak{b}}})}
      &w\!\circledast\! \bigl(z_{\mathfrak{b}}\!\circledast_{\mathfrak{b}}\! \text{ }\!\!_{\mathfrak{b}}\mathcal{F}(b_{\mathfrak{b}})\!\bigr)
      \ar[d]^-{\I_w\circledast \lambda^{\mathcal{F}}_{z_{\mathfrak{b}}}}
      \\
      &\text{ }
      &(w\!\circledast\! z)\!\circledast\! \mathcal{F}(b_{\mathfrak{b}})
      \ar@/^0.5pc/@{->>}[dl]|-{\mathit{cq}_{w\circledast z_{\mathfrak{b}},\mathcal{F}(b_{\mathfrak{b}})}}
      \ar@/_0.5pc/[dr]|-{t^{\mathcal{F}}_{w\circledast z,b_{\mathfrak{b}}}}
      &\mathcal{F}\bigl(w\!\circledast\! (z\!\circledast\! b_{\mathfrak{b}})\bigr)
      \ar[d]^-{\mathcal{F}(a_{w,z,b_{\mathfrak{b}}})}_-{\cong}
      \ar@/^1.5pc/[ddr]|-{\mathcal{F}(\I_w\circledast \gamma_{z_{\mathfrak{b}}})}
      &w\!\circledast\! \mathcal{F}(z_{\mathfrak{b}})
      \ar[dd]^-{t^{\mathcal{F}}_{w,z_{\mathfrak{b}}}}
      \\
      &(w\!\circledast\! z_{\mathfrak{b}})\!\circledast_{\mathfrak{b}}\! \text{ }\!\!_{\mathfrak{b}}\mathcal{F}(b_{\mathfrak{b}})
      \ar[d]^-{\lambda^{\mathcal{F}}_{w\circledast z_{\mathfrak{b}}}}
      &\text{ }
      &\mathcal{F}\bigl((w\!\circledast\! z)\!\circledast\! b_{\mathfrak{b}}\bigr)
      \ar[d]^-{\mathcal{F}(\gamma_{w\circledast z_{\mathfrak{b}}})}
      &\text{ }
      \\
      &\mathcal{F}(w\!\circledast\! z_{\mathfrak{b}})
      \ar@{=}[rr]
      &\text{ }
      &\mathcal{F}(w\!\circledast\! z_{\mathfrak{b}})
      \ar@{=}[r]
      &\mathcal{F}(w\!\circledast\! z_{\mathfrak{b}})
    }
  \end{equation*}
  After right-cancelling the epimorphism
  $\textcolor{teal}{\I_w\circledast \mathit{cq}_{z_{\mathfrak{b}},\text{ }\!\!_{\mathfrak{b}}\mathcal{F}(b_{\mathfrak{b}})}}$
  in the above diagram, we obtain the relation (\ref{eq2 EWthm lambdaF}).
  Thus we have a well-defined $\mathlcal{C}$-enriched natural transformation $\lambda^{\mathcal{F}}$
  as we claimed.
  From the definition of
  $\imath^{\mathfrak{b}}_{\text{ }\!\!_{\mathfrak{b}}\mathcal{F}(b_{\mathfrak{b}})}$
  given in (\ref{eq LeftModuleObj imathbbX}),
  we obtain that
  $\lambda^{\mathcal{F}}_{b_{\mathfrak{b}}}
  =\imath^{\mathfrak{b}}_{\text{ }\!\!_{\mathfrak{b}}\mathcal{F}(b_{\mathfrak{b}})}:
  b_{\mathfrak{b}}\circledast_{\mathfrak{b}} \text{ }\!\!_{\mathfrak{b}}\mathcal{F}(b_{\mathfrak{b}})
  \xrightarrow[]{\cong}\mathcal{F}(b_{\mathfrak{b}})$.
\end{proof}

\begin{proposition} \label{prop EWthm lambdaF}
  For each $\mathlcal{C}$-enriched functor
  $\mathcal{F}:\mathpzc{Mod}\!_{\mathfrak{b}}\rightarrow\mathcal{D}$,
  the following are equivalent:
  \begin{itemize}
    \item[(i)]
    $\mathcal{F}:\mathpzc{Mod}\!_{\mathfrak{b}}\rightarrow\mathcal{D}$ is a $\mathlcal{C}$-enriched left adjoint;

    \item[(ii)]
    $\mathcal{F}:\mathpzc{Mod}\!_{\mathfrak{b}}\rightarrow\mathcal{D}$
    is $\mathlcal{C}$-enriched cocontinuous;

    \item[(iii)]
    $\mathcal{F}:\mathpzc{Mod}\!_{\mathfrak{b}}\rightarrow\mathcal{D}$ 
    preserves $\mathlcal{C}$-tensors,
    and the underlying functor $\mathcal{F}_0$
    preserves coequalizers;

    \item[(iv)]
    The $\mathlcal{C}$-enriched natural transformation
    $\lambda^{\mathcal{F}}:-\circledast_{\mathfrak{b}}\text{ }\!\!_{\mathfrak{b}}\mathcal{F}(b_{\mathfrak{b}})\Rightarrow \mathcal{F}:\mathpzc{Mod}\!_{\mathfrak{b}}\rightarrow \mathcal{D}$
    defined in Lemma~\ref{lem EWthm lambdaF}
    is invertible.
  \end{itemize}
\end{proposition}
\begin{proof}      
  By Proposition~\ref{prop LeftModuleObj bX Cadj}, 
  (iv) implies (i).
  It is straightforward that (i) implies (ii),
  and (ii) implies (iii).
  We claim that (iii) implies (iv).
  Assume that the $\mathlcal{C}$-enriched functor
  $\mathcal{F}:\mathpzc{Mod}\!_{\mathfrak{b}}\rightarrow\mathcal{D}$
  preserves $\mathlcal{C}$-tensors,
  and the underlying functor
  $\mathcal{F}_0:(\mathpzc{Mod}\!_{\mathfrak{b}})_0\to \mathcal{D}_0$ preserves coequalizers.
  Recall the coequalizer diagram (\ref{eq TenEnCat gammazcoequalizer}) in $(\mathpzc{Mod}\!_{\mathfrak{b}})_0$.
  If we look at the diagram in (\ref{eq EWthm lambdaF}) we see that
  the top, middle horizontal morphisms in $\mathcal{D}$ are isomorphisms,
  and
  the right vertical morphisms also form a coequalizer diagram in the underlying category $\mathcal{D}_0$ of $\mathcal{D}$.
  This shows that $\lambda^{\mathcal{F}}_{z_{\mathfrak{b}}}$
  is an isomorphism in $\mathcal{D}$ for every $z_{\mathfrak{b}}\in\Obj(\mathpzc{Mod}\!_{\mathfrak{b}})$.
  We conclude that the $\mathlcal{C}$-enriched natural transformation $\lambda^{\mathcal{F}}$ is invertible.
\end{proof}

\begin{lemma} \label{lem EWthm lambdaF naturalinF}
  Let 
  $\mathcal{F}$, $\widetilde{\mathcal{F}}:\mathpzc{Mod}\!_{\mathfrak{b}}\rightarrow\mathcal{D}$
  be $\mathlcal{C}$-enriched functors.
  For each $\mathlcal{C}$-enriched natural transformation
  $\xi:\mathcal{F}\Rightarrow\widetilde{\mathcal{F}}:\mathpzc{Mod}\!_{\mathfrak{b}}\rightarrow\mathcal{D}$,
  we have the following relation for every $z_{\mathfrak{b}}\in \Obj(\mathpzc{Mod}\!_{\mathfrak{b}})$.
  \begin{equation} \label{eq EWthm lambdaF naturalinF}
    \vcenter{\hbox{
      \xymatrix@R=20pt@C=35pt{
        z_{\mathfrak{b}}\!\circledast_{\mathfrak{b}}\! \text{ }\!\!_{\mathfrak{b}}\mathcal{F}(b_{\mathfrak{b}})
        \ar[d]_-{\I_{z_{\mathfrak{b}}}\circledast_{\mathfrak{b}}\xi_{b_{\mathfrak{b}}}}
        \ar[r]^-{\lambda^{\mathcal{F}}_{z_{\mathfrak{b}}}}
        &\mathcal{F}(z_{\mathfrak{b}})
        \ar[d]^-{\xi_{z_{\mathfrak{b}}}}
        \\
        z_{\mathfrak{b}}\!\circledast_{\mathfrak{b}}\! \text{ }\!\!_{\mathfrak{b}}\widetilde{\mathcal{F}}(b_{\mathfrak{b}})
        \ar[r]^-{\lambda^{\widetilde{\mathcal{F}}}_{z_{\mathfrak{b}}}}
        &\widetilde{\mathcal{F}}(z_{\mathfrak{b}})
      }
    }}
  \end{equation}
\end{lemma}
\begin{proof}
  For each $z_{\mathfrak{b}}=(z,\gamma_z)\in\Obj(\mathpzc{Mod}\!_{\mathfrak{b}})$,
  we have the following diagram.
  \begin{equation*}
    \hspace*{-1.1cm}
    \xymatrix@R=15pt@C=35pt{
      &z\!\circledast\! \mathcal{F}(b_{\mathfrak{b}})
      \ar@{->>}[d]^-{\textcolor{teal}{\mathit{cq}_{z_{\mathfrak{b}},\text{ }\!\!_{\mathfrak{b}}\mathcal{F}(b_{\mathfrak{b}})}}}
      \ar@{=}[r]
      &z\!\circledast\! \mathcal{F}(b_{\mathfrak{b}})
      \ar[d]^-{t^{\mathcal{F}}_{z,b_{\mathfrak{b}}}}
      \ar@{=}[r]
      &z\!\circledast\! \mathcal{F}(b_{\mathfrak{b}})
      \ar[d]^-{\I_z\circledast \xi_{b_{\mathfrak{b}}}}
      \ar@{=}[r]
      &z\!\circledast\! \mathcal{F}(b_{\mathfrak{b}})
      \ar@{->>}[d]^-{\textcolor{teal}{\mathit{cq}_{z_{\mathfrak{b}},\text{ }\!\!_{\mathfrak{b}}\mathcal{F}(b_{\mathfrak{b}})}}}
      \\
      &z_{\mathfrak{b}}\!\circledast_{\mathfrak{b}}\! \text{ }\!\!_{\mathfrak{b}}\mathcal{F}(b_{\mathfrak{b}})
      \ar[d]^-{\lambda^{\mathcal{F}}_{z_{\mathfrak{b}}}}
      &\mathcal{F}(z\!\circledast\! b_{\mathfrak{b}})
      \ar@/^0.5pc/[dl]|-{\mathcal{F}(\gamma_{z_{\mathfrak{b}}})}
      \ar[d]^-{\xi_{z\circledast b_{\mathfrak{b}}}}
      &z\!\circledast\! \widetilde{\mathcal{F}}(b_{\mathfrak{b}})
      \ar@/^0.5pc/[dl]|-{t^{\widetilde{\mathcal{F}}}_{z,b_{\mathfrak{b}}}}
      \ar@/_0.5pc/@{->>}[dr]|-{\mathit{cq}_{z_{\mathfrak{b}},\text{ }\!\!_{\mathfrak{b}}\widetilde{\mathcal{F}}(b_{\mathfrak{b}})}}
      &z_{\mathfrak{b}}\!\circledast_{\mathfrak{b}}\! \text{ }\!\!_{\mathfrak{b}}\mathcal{F}(b_{\mathfrak{b}})
      \ar[d]^-{\I_{z_{\mathfrak{b}}}\circledast_{\mathfrak{b}} \xi_{b_{\mathfrak{b}}}}
      \\
      &\mathcal{F}(z_{\mathfrak{b}})
      \ar[d]^-{\xi_{z_{\mathfrak{b}}}}
      &\widetilde{\mathcal{F}}(z\!\circledast\! b_{\mathfrak{b}})
      \ar[d]^-{\widetilde{\mathcal{F}}(\gamma_{z_{\mathfrak{b}}})}
      &\text{ }
      &z_{\mathfrak{b}}\!\circledast_{\mathfrak{b}}\! \text{ }\!\!_{\mathfrak{b}}\widetilde{\mathcal{F}}(b_{\mathfrak{b}})
      \ar[d]^-{\lambda^{\widetilde{\mathcal{F}}}_{z_{\mathfrak{b}}}}
      \\
      &\widetilde{\mathcal{F}}(z_{\mathfrak{b}})
      \ar@{=}[r]
      &\widetilde{\mathcal{F}}(z_{\mathfrak{b}})
      \ar@{=}[rr]
      &\text{ }
      &\widetilde{\mathcal{F}}(z_{\mathfrak{b}})
    }
  \end{equation*}
  After right-cancelling the epimorphism
  $\textcolor{teal}{\mathit{cq}_{z_{\mathfrak{b}},\text{ }\!\!_{\mathfrak{b}}\mathcal{F}(b_{\mathfrak{b}})}}$
  in the above diagram, we obtain the relation (\ref{eq EWthm lambdaF naturalinF}).
\end{proof}

Let $\text{ }\!\!_{\mathfrak{b}}X$ be a left $\mathfrak{b}$-module object in $\mathcal{D}$.
The functor
$\mathlcal{C}\text{-}\mathit{Funct}(\text{$\mathpzc{Mod}\!_{\mathfrak{b}}$},\mathcal{D})
\to \text{ }\!\!_{\mathfrak{b}}\mathcal{D}$
of evaluating at $b_{\mathfrak{b}}\in\Obj(\mathpzc{Mod}\!_{\mathfrak{b}})$
defined in (\ref{eq EWthm rightCadj}) sends the $\mathlcal{C}$-enriched functor
$-\circledast_{\mathfrak{b}}\text{ }\!\!_{\mathfrak{b}} X:\mathpzc{Mod}\!_{\mathfrak{b}}\rightarrow\mathcal{D}$
to the left $\mathfrak{b}$-module object 
$\text{ }\!\!_{\mathfrak{b}}b_{\mathfrak{b}}\circledast_{\mathfrak{b}} \text{ }\!\!_{\mathfrak{b}}X
=(
  b_{\mathfrak{b}}\circledast_{\mathfrak{b}} \text{ }\!\!_{\mathfrak{b}}X,
  \rho_{b_{\mathfrak{b}}\circledast_{\mathfrak{b}}\text{ }\!\!_{\mathfrak{b}}X}
)$
in $\mathcal{D}$, where
\begin{equation*}
  \rho_{b_{\mathfrak{b}}\!\circledast_{\mathfrak{b}} \text{ }\!\!_{\mathfrak{b}}X}:\!\!
  \xymatrix@C=40pt{
    b\!\circledast\! (b_{\mathfrak{b}}\!\circledast_{\mathfrak{b}}\! \text{ }\!\!_{\mathfrak{b}}X)
    \ar[r]^-{a_{b,b_{\mathfrak{b}},\text{ }\!\!_{\mathfrak{b}}X}}_-{\cong}
    &(b\!\circledast\! b_{\mathfrak{b}})\!\circledast_{\mathfrak{b}}\! \text{ }\!\!_{\mathfrak{b}}X
    \ar[r]^-{\gamma_{b_{\mathfrak{b}}}\circledast_{\mathfrak{b}}\I_{\text{ }\!\!_{\mathfrak{b}}X}}
    &b_{\mathfrak{b}}\!\circledast_{\mathfrak{b}}\! \text{ }\!\!_{\mathfrak{b}}X
    .
  }
\end{equation*}
One can check that the isomorphism
$\imath^{\mathfrak{b}}_{\text{ }\!\!_{\mathfrak{b}}X}:b_{\mathfrak{b}}\circledast_{\mathfrak{b}}\text{ }\!\!_{\mathfrak{b}}X \xrightarrow[]{\cong} X$
in $\mathcal{D}$ defined in (\ref{eq LeftModuleObj imathbbX})
becomes an isomorphism
$\imath^{\mathfrak{b}}_{\text{ }\!\!_{\mathfrak{b}}X}:
\text{ }\!\!_{\mathfrak{b}}b_{\mathfrak{b}}\!\circledast_{\mathfrak{b}}\! \text{ }\!\!_{\mathfrak{b}}X
\xrightarrow[]{\cong}\text{ }\!\!_{\mathfrak{b}}X$
in $\text{ }\!\!_{\mathfrak{b}}\mathcal{D}$.
We are ready to prove Theorem~\ref{thm Intro EWthm}.

\begin{proof}[of Theorem~\ref{thm Intro EWthm}]
From the equivalence of statements \textit{(i)}-\textit{(iv)} in Proposition~\ref{prop EWthm lambdaF},
we conclude that the map (\ref{eq LeftModuleObj Leftadj})
induces an equivalence of categories
\begin{equation*}
\xymatrix{
\text{ }\!\!_{\mathfrak{b}}\mathcal{D}
\ar[r]^-{\simeq}
&\mathlcal{C}\text{-}\mathit{Funct}_{\mathit{cocon}}(\text{$\mathpzc{Mod}\!_{\mathfrak{b}}$},\mathcal{D})
}
\end{equation*}
between $\text{ }\!\!_{\mathfrak{b}}\mathcal{D}$
and the category of cocontinuous $\mathlcal{C}$-enriched functors
$\mathpzc{Mod}\!_{\mathfrak{b}}\to\mathcal{D}$.
The latter is a full, coreflective subcategory
 of the category
$\mathlcal{C}\text{-}\mathit{Funct}(\text{$\mathpzc{Mod}\!_{\mathfrak{b}}$},\mathcal{D})$
of all $\mathlcal{C}$-enriched functors
$\mathpzc{Mod}\!_{\mathfrak{b}}\to\mathcal{D}$
thanks to Lemma~\ref{lem EWthm lambdaF naturalinF}
by taking $\widetilde{\mathcal{F}}$ to be a general $\mathlcal{C}$-enriched functor
and $\mathcal{F}$ a $\mathlcal{C}$-cocontinuous one.
This completes the proof of Theorem~\ref{thm Intro EWthm}.

One can also directly show that the functor
  $\text{$\text{ }\!\!_{\mathfrak{b}}\mathcal{D}$}
  \rightarrow \mathlcal{C}\text{-}\mathit{Funct}(\mathpzc{Mod}\!_{\mathfrak{b}},\mathcal{D})$
  in (\ref{eq LeftModuleObj Leftadj}) is left adjoint to the functor
  $\mathlcal{C}\text{-}\mathit{Funct}(\text{$\mathpzc{Mod}\!_{\mathfrak{b}}$},\mathcal{D})
  \rightarrow \text{ }\!\!_{\mathfrak{b}}\mathcal{D}$
  in (\ref{eq EWthm rightCadj}).
    The component of the unit at
    each object $\text{ }\!\!_{\mathfrak{b}}X$ in $\text{ }\!\!_{\mathfrak{b}}\mathcal{D}$
    is the isomorphism 
    $(\imath^{\mathfrak{b}}_{\text{ }\!\!_{\mathfrak{b}}X})^{-1}
    \!:\!\!\!
    \xymatrix@C=15pt{
      \text{ }\!\!_{\mathfrak{b}}X
      \ar[r]^-{\cong}
      &\text{ }\!\!_{\mathfrak{b}}b_{\mathfrak{b}}\!\circledast_{\mathfrak{b}}\! \text{ }\!\!_{\mathfrak{b}}X
    }$
    in $\text{ }\!\!_{\mathfrak{b}}\mathcal{D}$.
    The component of the counit at
    each $\mathlcal{C}$-enriched functor
    $\mathcal{F}:\mathpzc{Mod}\!_{\mathfrak{b}}\rightarrow\mathcal{D}$
    is the $\mathlcal{C}$-enriched natural transformation
    $\lambda^{\mathcal{F}}:
-\!\circledast_{\mathfrak{b}}\! \text{ }\!\!_{\mathfrak{b}}\mathcal{F}(b_{\mathfrak{b}})
\Rightarrow\mathcal{F}$
    defined in Lemma~\ref{lem EWthm lambdaF}.
    One can check that the isomorphism 
    $(\imath^{\mathfrak{b}}_{\text{ }\!\!_{\mathfrak{b}}X})^{-1}$
    is natural in variable $\text{ }\!\!_{\mathfrak{b}}X$,
    and by Lemma~\ref{lem EWthm lambdaF naturalinF}
    $\lambda^{\mathcal{F}}$ is natural in variable $\mathcal{F}$.
    We can check the triangular identities using the relation
    $\lambda^{\mathcal{F}}_{b_{\mathfrak{b}}}=\imath^{\mathfrak{b}}_{\text{ }\!\!_{\mathfrak{b}}\mathcal{F}(b_{\mathfrak{b}})}$
and the explicit description
of $(\imath^{\mathfrak{b}}_{\text{ }\!\!_{\mathfrak{b}}X})^{-1}$
given in (\ref{eq LeftModuleObj imathbbX}).
The rest of the statements in Theorem~\ref{thm Intro EWthm}
are straightforward to check using Proposition~\ref{prop EWthm lambdaF}.
This is another proof of Theorem~\ref{thm Intro EWthm}.
\end{proof}

\begin{proof}[of Corollary~\ref{cor Intro EWthm}]
  Let $\mathfrak{b}'$ be another monoid in $\mathlcal{C}$.
  After substituting $\mathcal{D}=\mathpzc{Mod}\!_{\mathfrak{b}'}$
  in Theorem~\ref{thm Intro EWthm},
  we obtain the adjoint equivalence of categories
  \begin{equation*}
    \xymatrix@R=0pt{
      \text{ }\!\!_{\mathfrak{b}}\mathpzc{Mod}\!_{\mathfrak{b}'}
      \ar@<0.5ex>[r]^-{\simeq}
      &{\mathlcal{C}\text{-}\mathit{Funct}_{\mathit{cocon}}(\mathpzc{Mod}\!_{\mathfrak{b}},\mathpzc{Mod}\!_{\mathfrak{b}'})}
      \ar@<0.5ex>[l]^-{\simeq}
    }
  \end{equation*}
  whose right adjoint is the functor of evaluating at $b_{\mathfrak{b}}$.
\end{proof}

\section{Morita Theory} \label{sec Morita}
In this section, we prove Theorem~\ref{thm Intro CharacterizingMb}
which characterizes when a $\mathlcal{C}$-enriched category $\mathcal{D}$
is equivalent to $\mathpzc{Mod}\!_{\mathfrak{b}}$ for a given monoid $\mathfrak{b}$ in $\mathlcal{C}$.
We also give a proof of Corollary~\ref{cor Intro CosMorita}
which generalizes the result of Morita in enriched context.

\begin{definition} \label{def Morita Enrichedcompactgen}
  Let $\mathcal{D}$ be a $\mathlcal{C}$-enriched category
  and let $X\in\Obj(\mathcal{D})$.
  We say
  \begin{itemize}
    \item[(i)]
    $X$ is a \emph{$\mathlcal{C}$-enriched compact object} in $\mathcal{D}$
    if the $\mathlcal{C}$-enriched Hom functor
    $\mathcal{D}(X,-):\mathcal{D}\rightarrow \mathcal{C}$
    preserves $\mathlcal{C}$-tensors, and
    the underlying functor
    $\mathcal{D}(X,-)_0$
    preserves coequalizers;
    
    \item[(ii)]
    $X$ is a \emph{$\mathlcal{C}$-enriched generator} in $\mathcal{D}$
    if the $\mathlcal{C}$-enriched Hom functor
    $\mathcal{D}(X,-):\mathcal{D}\rightarrow \mathcal{C}$
    is conservative;
  
    \item[(iii)]
    $X$ is a \emph{$\mathlcal{C}$-enriched compact generator} in $\mathcal{D}$
    if it is both a $\mathlcal{C}$-enriched compact object and a $\mathlcal{C}$-enriched generator in $\mathcal{D}$.
  \end{itemize}    
\end{definition}

\begin{example} \label{example}
  Consider the case when $\mathlcal{C}=\mathlcal{Ab}$
  is the closed symmetric monoidal category of abelian groups.
  Let $R$ be a ring and let $\mathpzc{Mod}\!_R$ be the preadditive category of right $R$-modules.
  For each right $R$-module $N_R$,
  \begin{itemize}
    \item[(i)]
    $N_R$ is an $\mathlcal{Ab}$-enriched compact object in $\mathpzc{Mod}\!_R$
    if and only if it is a finitely generated projective right $R$-module;
    
    \item[(ii)]
    $N_R$ is an $\mathlcal{Ab}$-enriched generator in $\mathpzc{Mod}\!_R$
    if and only if it is a generator in the category of right $R$-modules;
    
    \item[(iii)]
    $N_R$ is an $\mathlcal{Ab}$-enriched compact generator in $\mathpzc{Mod}\!_R$
    if and only if it is a finitely generated projective generator
    in the category of right $R$-modules.
  \end{itemize}
Let us explain the `only if' part of statement (i).
Assume that $N_R$ is an $\mathlcal{Ab}$-enriched compact object in $\mathpzc{Mod}\!_{R}$.
By Proposition~\ref{prop EWthm lambdaF},
the $\mathlcal{Ab}$-enriched Hom functor 
$\mathpzc{Mod}\!_{R}(N_R,-)$ is $\mathlcal{Ab}$-enriched cocontinuous.
In particular, the underlying functor
$\mathpzc{Mod}\!_{R}(N_R,-)_0$
is cocontinuous.
\begin{itemize}
\item
$N_R$ is a projective right $R$-module if and only if the underlying functor
$\mathpzc{Mod}\!_{R}(N_R,-)_0$ preserves coequalizers.

\item
A projective right $R$-module $N_R$
is finitely generated if and only if the underlying functor
$\mathpzc{Mod}\!_{R}(N_R,-)_0$ preserves arbitrary sums.
This is explained in the proof of \cite[Proposition1.2(c)]{Bass1968}.
\end{itemize}
Therefore $N_R$ is a finitely generated projective right $R$-module.
\end{example}

\begin{lemma}\label{lem Morita bbcompactgen}
  Let $\mathfrak{b}=(b,u_b,m_b)$ be a monoid in $\mathlcal{C}$.
  The right $\mathfrak{b}$-module
  $b_{\mathfrak{b}}$ is a $\mathlcal{C}$-enriched compact generator in $\mathpzc{Mod}\!_{\mathfrak{b}}$,
  and 
  we have an isomorphism of monoids
  $\text{$\mathfrak{b}$}\cong\mathit{End}_{\mathpzc{Mod}\!_{\mathfrak{b}}}(b_{\mathfrak{b}})$
  in $\mathlcal{C}$.    
\end{lemma}
\begin{proof}
  Recall that for each $z_{\mathfrak{b}}\in\Obj(\mathpzc{Mod}\!_{\mathfrak{b}})$,
  we have a morphism 
  $\gamma_{z_{\mathfrak{b}}}:z\circledast b_{\mathfrak{b}}\rightarrow z_{\mathfrak{b}}$
  in $\mathpzc{Mod}\!_{\mathfrak{b}}$.
  One can check that the corresponding right adjunct
    $\text{$\bar{\gamma}_{z_{\mathfrak{b}}}$}:\text{ }z\xrightarrow[]{\cong} \mathpzc{Mod}\!_{\mathfrak{b}}(b_{\mathfrak{b}},z_{\mathfrak{b}})$
  is an isomorphism in $\mathcal{C}$,
  and is $\mathlcal{C}$-enriched natural in variable $z_{\mathfrak{b}}$.
  Thus we have an isomorphism of $\mathlcal{C}$-enriched functors
  $\mathcal{U}\cong \mathpzc{Mod}\!_{\mathfrak{b}}(b_{\mathfrak{b}},-):\mathpzc{Mod}\!_{\mathfrak{b}}\rightarrow\mathcal{C}$.
  The forgetful $\mathlcal{C}$-enriched functor
  $\mathcal{U}:\mathpzc{Mod}\!_{\mathfrak{b}}\rightarrow\mathcal{C}$ is conservative,
  preserves $\mathlcal{C}$-tensors,
  and its underlying functor
  $\mathcal{U}_0$ preserves coequalizers.
  We conclude that $b_{\mathfrak{b}}$ is a $\mathlcal{C}$-enriched compact generator in $\mathpzc{Mod}\!_{\mathfrak{b}}$.
  We leave for the readers to check that the isomorphism
  $\text{$\bar{\gamma}_{b_{\mathfrak{b}}}$}:b\xrightarrow[]{\cong} \mathpzc{Mod}\!_{\mathfrak{b}}(b_{\mathfrak{b}},b_{\mathfrak{b}})$
  in $\mathcal{C}$
  becomes an isomorphism of monoids
  $\text{$\bar{\gamma}_{b_{\mathfrak{b}}}:\mathfrak{b}$}\cong\mathit{End}_{\mathpzc{Mod}\!_{\mathfrak{b}}}(b_{\mathfrak{b}})$
  in $\mathlcal{C}$.
\end{proof}

We are ready to prove Theorem~\ref{thm Intro CharacterizingMb}.

\begin{proof}[of Theorem~\ref{thm Intro CharacterizingMb}]
  By Lemma~\ref{lem Morita bbcompactgen}, the only if part is true.
  We prove the if part as follows.
  Let us denote
  $f:b\xrightarrow[]{\cong}\mathcal{D}(X,X)$
  as the isomorphism in $\mathcal{C}$.
  Then we have a morphism
    $\rho_X:
    \!\!
    \xymatrix@C=25pt{
      b\!\circledast\! X
      \ar[r]^-{f\circledast\I_X}_-{\cong}
      &\mathcal{D}(X,X)\!\circledast\! X
      \ar[r]^-{\mathit{Ev}_{X,X}}
      &X
    }$
  in $\mathcal{D}$
  whose right adjunct is $\bar{\rho}_X=f:b\xrightarrow[]{\cong}\mathcal{D}(X,X)$,
  and the pair $\text{ }\!\!_{\mathfrak{b}}X=(X,\rho_X)$
  is a left $\mathfrak{b}$-module object in $\mathcal{D}$.
  By Proposition~\ref{prop LeftModuleObj bX Cadj}, we have the following
  adjoint pair of $\mathlcal{C}$-enriched functors.
  \begin{equation} \label{eq Intro CharacterizingMb proof}
    \xymatrix@C=60pt{
      \mathpzc{Mod}\!_{\mathfrak{b}}
      \ar@<1ex>[r]^-{\alpha\text{ }\!:=\text{ }\!-\text{ }\!\circledast_{\mathfrak{b}}\text{ }\!\!_{\mathfrak{b}}X}
      &\mathcal{D}
      \ar@<1ex>[l]^-{\beta\text{ }\!:=\text{ }\!\mathcal{D}(\text{ }\!\!_{\mathfrak{b}}X,-)}
    }
  \end{equation}
  We are going to show that the $\mathlcal{C}$-enriched adjunction (\ref{eq Intro CharacterizingMb proof})
  is an adjoint equivalence of $\mathlcal{C}$-enriched categories.
  First, we show that
  $\beta\alpha:\mathpzc{Mod}\!_{\mathfrak{b}}\rightarrow\mathpzc{Mod}\!_{\mathfrak{b}}$
  is $\mathlcal{C}$-enriched cocontinuous as follows.
  Recall the diagram in (\ref{eq LeftModuleObj T(bX,-)}).
  \begin{itemize}
    \item 
    The $\mathlcal{C}$-enriched functor
    $\mathcal{D}(X,-):\mathcal{D}\rightarrow\mathcal{C}$
    preserves $\mathlcal{C}$-tensors,
    and the underlying functor $\mathcal{D}(X,-)_0$ preserves coequalizers.
  
    \item
    The $\mathlcal{C}$-enriched category $\mathpzc{Mod}\!_{\mathfrak{b}}$ is tensored,
    and the underlying category $(\mathpzc{Mod}\!_{\mathfrak{b}})_0$ has coequalizers.
  
    \item
    The forgetful $\mathlcal{C}$-enriched functor
    $\mathcal{U}:\mathpzc{Mod}\!_{\mathfrak{b}}\rightarrow\mathcal{C}$ 
    is conservative,
    preserves $\mathlcal{C}$-tensors,
    and the underlying functor $\mathcal{U}_0$ preserves coequalizers.  
  \end{itemize}
  Thus we obtain that the $\mathlcal{C}$-enriched functor
  $\beta=\mathcal{D}(\text{ }\!\!_{\mathfrak{b}}X,-):\mathcal{D}\rightarrow\mathpzc{Mod}\!_{\mathfrak{b}}$
  preserves $\mathlcal{C}$-tensors,
  and the underlying functor $\beta_0$
  preserves coequalizers.
  Then the $\mathlcal{C}$-enriched functor
  $\beta\alpha:\mathpzc{Mod}\!_{\mathfrak{b}}\rightarrow\mathpzc{Mod}\!_{\mathfrak{b}}$
  also has the same properties.
  By Proposition~\ref{prop EWthm lambdaF},
  we conclude that the $\mathlcal{C}$-enriched functor
  $\beta\alpha:\mathpzc{Mod}\!_{\mathfrak{b}}\rightarrow\mathpzc{Mod}\!_{\mathfrak{b}}$
  is cocontinuous.
  
  Next, we show that the adjunction (\ref{eq Intro CharacterizingMb proof})
  is an adjoint equivalence of $\mathlcal{C}$-enriched categories.
  We begin by showing that the unit
  $\eta:I_{\mathpzc{Mod}\!_{\mathfrak{b}}}\Rightarrow \beta\alpha:\mathpzc{Mod}\!_{\mathfrak{b}}\rightarrow\mathpzc{Mod}\!_{\mathfrak{b}}$
  is a $\mathlcal{C}$-enriched natural isomorphism.
  By Corollary~\ref{cor Intro EWthm},
  it suffices to show that the component
  $\eta_{b_{\mathfrak{b}}}:b_{\mathfrak{b}}\rightarrow \mathcal{D}(\text{ }\!\!_{\mathfrak{b}}X,b_{\mathfrak{b}}\circledast_{\mathfrak{b}} \text{ }\!\!_{\mathfrak{b}}X)$
  at $b_{\mathfrak{b}}$ is an isomorphism in $\mathpzc{Mod}\!_{\mathfrak{b}}$.
  Consider the following diagram.
  \begin{equation*}
    \hspace*{-0.7cm}
    \xymatrix@R=15pt@C=25pt{
      &b
      \ar@{=}[r]
      \ar[dd]^-{\eta_{b_{\mathfrak{b}}}}
      &b
      \ar[d]^-{\mathit{Cv}_{b,X}}
      \ar@{=}[rr]
      &\text{ }
      &b
      \ar[d]^-{f}_{\cong}
      \\
      &\text{ }
      &\mathcal{D}(X,b\!\circledast\! X)
      \ar@/^0.5pc/[dl]|-{(\mathit{cq}_{b_{\mathfrak{b}},\text{ }\!\!_{\mathfrak{b}}X})_{\star}}
      \ar[dd]^-{(\rho_X)_{\star}}
      \ar@/^0.5pc/[dr]^-{(f\circledast\I_X)_{\star}}_{\cong}
      &\text{ }
      &\mathcal{D}(X,X)
      \ar@/_0.5pc/[dl]_-{\mathit{Cv}_{\mathcal{D}(X,X),X}}
      \ar@{=}[dd]
      \\
      &\mathcal{D}(X,b_{\mathfrak{b}}\!\circledast_{\mathfrak{b}}\! \text{ }\!\!_{\mathfrak{b}}X)
      \ar[d]^-{(\imath^{\mathfrak{b}}_{\text{ }\!\!_{\mathfrak{b}}X})_{\star}}_-{\cong}
      &\text{ }
      &\mathcal{D}\bigl(X,\mathcal{D}(X,X)\!\circledast\! X\bigr)
      \ar[d]^-{(\mathit{Ev}_{X,X})_{\star}}
      &\text{ }
      \\
      &\mathcal{D}(X,X)
      \ar@{=}[r]
      &\mathcal{D}(X,X)
      \ar@{=}[r]
      &\mathcal{D}(X,X)
      \ar@{=}[r]
      &\mathcal{D}(X,X)
    }
  \end{equation*} 
  We obtain that the morphism
  $\eta_{b_{\mathfrak{b}}}:b\rightarrow \mathcal{D}(X, b_{\mathfrak{b}}\circledast_{\mathfrak{b}} \text{ }\!\!_{\mathfrak{b}}X)$ 
  in $\mathcal{C}$ is equal to
  $(\imath^{\mathfrak{b}}_{\text{ }\!\!_{\mathfrak{b}}X})^{-1}_{\star}\circ f
  :b\xrightarrow[]{\cong}\mathcal{D}(X,X)\xrightarrow[]{\cong}\mathcal{D}(X,b_{\mathfrak{b}}\circledast_{\mathfrak{b}} \text{ }\!\!_{\mathfrak{b}}X)$
  which is an isomorphism.
  This shows that the unit
  $\eta:I_{\mathpzc{Mod}\!_{\mathfrak{b}}}\Rightarrow \beta\alpha$
  is a $\mathlcal{C}$-enriched natural isomorphism.
  
  To conclude that the $\mathlcal{C}$-enriched adjunction
  (\ref{eq Intro CharacterizingMb proof}) is an equivalence of $\mathlcal{C}$-enriched categories,
  it suffices to show that the right adjoint
  $\beta=\mathcal{D}(\text{ }\!\!_{\mathfrak{b}}X,-):\mathcal{D}\to \mathpzc{Mod}\!_{\mathfrak{b}}$
  is conservative.
  This is because any $\mathlcal{C}$-enriched adjunction
  with fully faithful left adjoint and conservative right adjoint
  is an adjoint equivalence of $\mathlcal{C}$-enriched categories
  due to the triangular identities.
  As we assumed that $X$ is also a $\mathlcal{C}$-enriched generator in $\mathcal{D}$,
  the $\mathlcal{C}$-enriched functor $\mathcal{D}(X,-):\mathcal{D}\rightarrow\mathcal{C}$
  is conservative.
  From the relation (\ref{eq LeftModuleObj T(bX,-)}),
  we obtain that
  $\beta=\mathcal{D}(\text{ }\!\!_{\mathfrak{b}}X,-):\mathcal{D}\to \mathpzc{Mod}\!_{\mathfrak{b}}$
 is also conservative.
  This completes the proof of Theorem~\ref{thm Intro CharacterizingMb}.
\end{proof}

\begin{remark}
  Let us weaken the assumption of Theorem~\ref{thm Intro CharacterizingMb}
  and merely assume that $X$ is a $\mathlcal{C}$-enriched compact object in $\mathcal{D}$.
  Then the left adjoint $\mathlcal{C}$-enriched functor
  $\alpha:\mathpzc{Mod}\!_{\mathfrak{b}}\rightarrow\mathcal{D}$
  in (\ref{eq Intro CharacterizingMb proof})
  induces an equivalence of $\mathlcal{C}$-enriched categories
  from $\mathpzc{Mod}\!_{\mathfrak{b}}$
  to a coreflective full $\mathlcal{C}$-enriched subcategory of $\mathcal{D}$.
\end{remark}

\begin{remark}
 Theorem 1.3 is related to the result in \cite{Berger2019}
which states that the Eilenberg-Moore category of a $\mathlcal{C}$-enriched $\mathlcal{C}$-tensor preserving monad $\mathcal{T}$ on $\mathcal{C}$ is equivalent
 to the category of right $\mathcal{T}(c)$-modules.
\end{remark}

Let $\mathfrak{b}=(b,u_b,m_b)$ be a monoid in $\mathlcal{C}$.
We have a $\mathlcal{C}$-enriched natural isomorphism
\begin{equation}\label{eq Morita jmathb}
  \text{$\jmath^{\mathfrak{b}}:$}
  \xymatrix@C=18pt{
    -\!\circledast_{\mathfrak{b}}\! \text{ }\!\!_{\mathfrak{b}}b_{\mathfrak{b}}    
    \ar@2{->}[r]^-{\cong}
    &I_{\mathpzc{Mod}\!_{\mathfrak{b}}}:
    \mathpzc{Mod}\!_{\mathfrak{b}}\rightarrow\mathpzc{Mod}\!_{\mathfrak{b}}
  }    
\end{equation}
whose component at $z_{\mathfrak{b}}=(z,\gamma_z)\in\Obj(\mathpzc{Mod}\!_{\mathfrak{b}})$
is the unique isomorphism
$\jmath^{\mathfrak{b}}_{z_{\mathfrak{b}}}:z_{\mathfrak{b}}\circledast_{\mathfrak{b}}\text{ }\!\!_{\mathfrak{b}}b_{\mathfrak{b}}\xrightarrow[]{\cong}z_{\mathfrak{b}}$
in $\mathcal{D}$ satisfying the relation
\begin{equation*}
  \vcenter{\hbox{
    \xymatrix@R=15pt@C=50pt{
      z\!\circledast\! b_{\mathfrak{b}}
      \ar@{->>}[d]_-{\mathit{cq}_{z_{\mathfrak{b}},\text{ }\!\!_{\mathfrak{b}}b_{\mathfrak{b}}}}
      \ar@/^1pc/[dr]^-{\gamma_{z_{\mathfrak{b}}}}
      \\
      z_{\mathfrak{b}}\!\circledast_{\mathfrak{b}}\! \text{ }\!\!_{\mathfrak{b}}b_{\mathfrak{b}}
      \ar@{.>}[r]^(0.54){\exists!\text{ }\jmath^{\mathfrak{b}}_{z_{\mathfrak{b}}}}_-{\cong}
      &z_{\mathfrak{b}}
    }
  }}
  \qquad\quad
  \gamma_{z_{\mathfrak{b}}}=
  \jmath^{\mathfrak{b}}_{z_{\mathfrak{b}}}
  \circ
  \mathit{cq}_{z_{\mathfrak{b}},\text{ }\!\!_{\mathfrak{b}}b_{\mathfrak{b}}}
  .
\end{equation*}
Let $\mathfrak{b}'$, $\mathfrak{b}''$ be additional monoids in $\mathlcal{C}$.
For each pair of a $(\mathfrak{b},\mathfrak{b}')$-bimodule
$\text{ }\!\!_{\mathfrak{b}}x_{\mathfrak{b}'}=(x_{\mathfrak{b}'},\rho_{x_{\mathfrak{b}'}})$
and a $(\mathfrak{b}',\mathfrak{b}'')$-bimodule
$\text{ }\!\!_{\mathfrak{b}'}y_{\mathfrak{b}''}$,
we have the $(\mathfrak{b},\mathfrak{b}'')$-bimodule
\begin{equation*}
  \text{ }\!\!_{\mathfrak{b}}x_{\mathfrak{b}'}\!\circledast_{\mathfrak{b}'}\! \text{ }\!\!_{\mathfrak{b}'}y_{\mathfrak{b}''}
  =\bigl(
    x_{\mathfrak{b}'}\!\circledast_{\mathfrak{b}'}\! \text{ }\!\!_{\mathfrak{b}'}y_{\mathfrak{b}''}
    ,\text{ }\text{ }
    \rho_{x_{\mathfrak{b}'}\circledast_{\mathfrak{b}'} \text{ }\!\!_{\mathfrak{b}'}y_{\mathfrak{b}''}}\!:\!\!
    \xymatrix@C=17pt{
      b\!\circledast\! (x_{\mathfrak{b}'}\!\circledast_{\mathfrak{b}'}\! \text{ }\!\!_{\mathfrak{b}'}y_{\mathfrak{b}''})
      \ar[r]
      &x_{\mathfrak{b}'}\!\circledast_{\mathfrak{b}'}\! \text{ }\!\!_{\mathfrak{b}'}y_{\mathfrak{b}''}
    }
    \!
  \bigr)
\end{equation*}
whose left $\mathfrak{b}$-action is given by
\begin{equation*}
  \rho_{x_{\mathfrak{b}'}\circledast_{\mathfrak{b}'} \text{ }\!\!_{\mathfrak{b}'}y_{\mathfrak{b}''}}\!:\!\!
  \xymatrix@C=42pt{
    b\!\circledast\! (x_{\mathfrak{b}'}\!\circledast_{\mathfrak{b}'}\! \text{ }\!\!_{\mathfrak{b}'}y_{\mathfrak{b}''})
    \ar[r]^-{a_{\text{ }\!b\text{ }\!,\text{ }\!x_{\mathfrak{b}'},\text{ }\!\!_{\mathfrak{b}'}y_{\mathfrak{b}''}}}_-{\cong}
    &(b\!\circledast\! x_{\mathfrak{b}'})\!\circledast_{\mathfrak{b}'}\! \text{ }\!\!_{\mathfrak{b}'}y_{\mathfrak{b}''}
    \ar[r]^-{\rho_{x_{\mathfrak{b}'}}\!\circledast_{\mathfrak{b}'} \I_{\text{ }\!\!_{\mathfrak{b}'}y_{\mathfrak{b}''}}}
    &x_{\mathfrak{b}'}\!\circledast_{\mathfrak{b}'}\! \text{ }\!\!_{\mathfrak{b}'}y_{\mathfrak{b}''}
    .
  }
\end{equation*}
We have a $\mathlcal{C}$-enriched natural isomorphism 
\begin{equation} \label{eq Morita associatorisom}
 \text{$ a_{-,\text{ }\!\!_{\mathfrak{b}}x_{\mathfrak{b}}\text{ }\!\!\!\!_{\text{ }^{\pr}},\text{ }\!\!_{\mathfrak{b}}\text{ }\!\!\!\!_{\text{ }^{\pr}}y_{\mathfrak{b}}\text{ }\!\!\!\!_{\text{ }^{\pr\!\pr}}}:$}
  \xymatrix@C=20pt{
    -\!\circledast_{\mathfrak{b}}\! (\text{ }\!\!_{\mathfrak{b}}x_{\mathfrak{b}'}\!\circledast_{\mathfrak{b}'}\! \text{ }\!\!_{\mathfrak{b}'}y_{\mathfrak{b}''})
    \ar@2{->}[r]^-{\cong}
    &(-\!\circledast_{\mathfrak{b}}\! \text{ }\!\!_{\mathfrak{b}}x_{\mathfrak{b}'})\!\circledast_{\mathfrak{b}'}\! \text{ }\!\!_{\mathfrak{b}'}y_{\mathfrak{b}''}
  }
  :\mathpzc{Mod}\!_{\mathfrak{b}}\rightarrow\mathpzc{Mod}\!_{\mathfrak{b}''}
\end{equation}
whose component
$a_{z_{\mathfrak{b}},\text{ }\!\!_{\mathfrak{b}}x_{\mathfrak{b}}\text{ }\!\!\!\!_{\text{ }^{\pr}},\text{ }\!\!_{\mathfrak{b}}\text{ }\!\!\!\!_{\text{ }^{\pr}}y_{\mathfrak{b}}\text{ }\!\!\!\!_{\text{ }^{\pr\!\pr}}}$
at $z_{\mathfrak{b}}\in\Obj(\mathpzc{Mod}\!_{\mathfrak{b}})$
is the unique morphism in $\mathpzc{Mod}\!_{\mathfrak{b}''}$
which makes the following diagram commutative.
\begin{equation*}
  \xymatrix@R=18pt@C=70pt{
    z\!\circledast\! (x_{\mathfrak{b}'}\!\circledast_{\mathfrak{b}'}\! \text{ }\!\!_{\mathfrak{b}'}y_{\mathfrak{b}''})
    \ar@{->>}[d]_-{\mathit{cq}_{z_{\mathfrak{b}},\text{ }\!\!_{\mathfrak{b}}x_{\mathfrak{b}}\text{ }\!\!\!\!\!_{\text{ }^{\pr}}\circledast_{\mathfrak{b}}\text{ }\!\!\!\!\!_{\text{ }^{\pr}}\text{ }\!\!_{\mathfrak{b}}\text{ }\!\!\!\!\!_{\text{ }^{\pr}}y_{\mathfrak{b}}\text{ }\!\!\!\!\!_{\text{ }^{\pr\!\pr}}}}
    \ar[r]^-{a_{z,x_{\mathfrak{b}'},\text{ }\!\!_{\mathfrak{b}'}y_{\mathfrak{b}''}}}_-{\cong}
    &(z\!\circledast\! x_{\mathfrak{b}'})\!\circledast_{\mathfrak{b}'}\! \text{ }\!\!_{\mathfrak{b}'}y_{\mathfrak{b}''}
    \ar@{->>}[d]^-{\mathit{cq}_{z_{\mathfrak{b}},\text{ }\!\!_{\mathfrak{b}}x_{\mathfrak{b}}\text{ }\!\!\!\!\!_{\text{ }^{\pr}}}\circledast_{\mathfrak{b}'}\I_{\text{ }\!\!_{\mathfrak{b}'}y_{\mathfrak{b}''}}}
    \\
    z_{\mathfrak{b}}\!\circledast_{\mathfrak{b}}\! (\text{ }\!\!_{\mathfrak{b}}x_{\mathfrak{b}'}\!\circledast_{\mathfrak{b}'}\! \text{ }\!\!_{\mathfrak{b}'}y_{\mathfrak{b}''})
    \ar@{.>}[r]^-{\exists!\text{ }a_{z_{\mathfrak{b}},\text{ }\!\!_{\mathfrak{b}}x_{\mathfrak{b}}\text{ }\!\!\!\!\!_{\text{ }^{\pr}},\text{ }\!\!_{\mathfrak{b}}\text{ }\!\!\!\!\!_{\text{ }^{\pr}}y_{\mathfrak{b}}\text{ }\!\!\!\!\!_{\text{ }^{\pr\!\pr}}}}_-{\cong}
    &(z_{\mathfrak{b}}\!\circledast_{\mathfrak{b}}\! \text{ }\!\!_{\mathfrak{b}}x_{\mathfrak{b}'})\!\circledast_{\mathfrak{b}'}\! \text{ }\!\!_{\mathfrak{b}'}y_{\mathfrak{b}''}
  }
\end{equation*}
We are ready to prove Corollary~\ref{cor Intro CosMorita}.

\begin{proof}[of Corollary~\ref{cor Intro CosMorita}]
  By substituting $\mathcal{D}=\mathpzc{Mod}\!_{\mathfrak{b}'}$ in Theorem~\ref{thm Intro CharacterizingMb},
  we immediately obtain that statements (i), (ii) are equivalent.
  We are left to show that statements (i), (iii) are equivalent.  
  The monoids $\mathfrak{b}$, $\mathfrak{b}'$ in $\mathlcal{C}$
  are Morita equivalent if and only if
  there exist a pair of cocontinuous $\mathlcal{C}$-enriched functors
  $\alpha:\mathpzc{Mod}\!_{\mathfrak{b}}\rightarrow \mathpzc{Mod}\!_{\mathfrak{b}'}$,
  $\beta:\mathpzc{Mod}\!_{\mathfrak{b}'}\rightarrow \mathpzc{Mod}\!_{\mathfrak{b}}$
  together with a pair of $\mathlcal{C}$-enriched natural isomorphisms
  $\beta\alpha\cong I_{\mathpzc{Mod}\!_{\mathfrak{b}}}$,
  $\alpha\beta\cong I_{\mathpzc{Mod}\!_{\mathfrak{b}'}}$.
  By Corollary~\ref{cor Intro EWthm}
  and using the $\mathlcal{C}$-enriched natural isomorphisms
  (\ref{eq Morita jmathb}), (\ref{eq Morita associatorisom}),
  we obtain that the existence of such pair $\alpha$, $\beta$
  is equivalent to the existence of 
  bimodules $\text{ }\!\!_{\mathfrak{b}}x_{\mathfrak{b}'}$, $\text{ }\!\!_{\mathfrak{b}'}y_{\mathfrak{b}}$
  together with isomorphisms of bimodules
  $\text{ }\!\!_{\mathfrak{b}}x_{\mathfrak{b}'}\circledast_{\mathfrak{b'}}\text{ }\!\!_{\mathfrak{b}'}y_{\mathfrak{b}}\cong \text{ }\!\!_{\mathfrak{b}}b_{\mathfrak{b}}$
  and
  $\text{ }\!\!_{\mathfrak{b}'}y_{\mathfrak{b}}\circledast_{\mathfrak{b}}\text{ }\!\!_{\mathfrak{b}}x_{\mathfrak{b}'}\cong \text{ }\!\!_{\mathfrak{b}'}b'\text{ }\!\!\!\!_{\mathfrak{b}'}$.
\end{proof}

\end{document}